\documentclass[11pt]{article}

\usepackage[top=50mm, bottom=50mm, left=50mm, right=50mm]{geometry}
\usepackage{lineno}
\usepackage{amssymb}
\usepackage{amsmath}
\usepackage{amsthm}
\usepackage{dsfont}
\usepackage{epsfig}
\usepackage{graphicx}
\usepackage{graphics}
\usepackage{float}
\usepackage{subfigure}
\usepackage{multirow}
\usepackage{color}
\usepackage{lineno}
\usepackage{fullpage}
\usepackage[normalem]{ulem} 
\usepackage{makeidx}
\usepackage{xspace}
\usepackage{wrapfig}
\usepackage{mathtools}
\usepackage{authblk}
\usepackage{etoolbox}
\usepackage{truncate}
\usepackage{xkeyval}
\usepackage{xstring}
\usepackage{xcolor}
\usepackage{soul}
\sethlcolor{yellow}

\usepackage{hyperref}
\hypersetup{
colorlinks = true,
linkcolor = {red},
urlcolor = {blue},
citecolor = {green}
}

\usepackage[final]{changes}
\definechangesauthor[name={Quan}, color=orange]{quan}

\makeindex

\newtheorem{theorem}{Theorem}

\newtheorem{corollary}[theorem]{Corollary}
\newtheorem{proposition}[theorem]{Proposition}

\newtheorem{lemma}[theorem]{Lemma}

\newtheorem{example}{Example}

\newcommand{\Ebb}{\mathbb{E}}

\newcommand{\Rbb}{\mathbb{R}}

\newcommand{\Zbb}{\mathbb{Z}}
\newcommand{\Pbb}{\mathbb{P}}

\newcommand{\Ecal}{\mathcal{E}}
\newcommand{\Fcal}{\mathcal{F}}

\newcommand{\Lcal}{\mathcal{L}}
\newcommand{\Ocal}[1]{\mathcal{O}\left(#1\right)}

\newcommand{\drm}{\textrm{d}}

\newcommand\bigsubset[1][1.19]{%
   \mathrel{\vcenter{\hbox{\scalebox{#1}{$\subset$}}}}}

\newcommand{\ie}{i.e. }
\newcommand{\eg}{e.g. }

\newcommand{\ddt}{\dfrac{\drm}{\drm t}}
\newcommand{\esperance}[1]{\Ebb\left[ #1 \right]}
\newcommand{\variance}[1]{\textrm{Var}\left[ #1 \right]}
\newcommand{\esperancewithstartingpoint}[2]{\Ebb_{#1}\left[ #2 \right]}

\newcommand{\norminfty}[1]{\left\| #1 \right\|_\infty}
\newcommand{\proba}[1]{\Pbb\left[ #1 \right]}
\newcommand{\probawithstartingpoint}[2]{\Pbb_{#1}\left[ #2 \right]}
\newcommand{\indicator}[1]{\mathds{1}_{\left\{#1\right\}}}

\newcommand{\tmixxepsilon}{\textrm{t}_\textsc{mix}(x;\epsilon)}

\newcommand{\crochet}[2]{\left<#1, #2\right>}
\newcommand{\restr}[2]{{
  \left.\kern-\nulldelimiterspace 
  #1 
  \vphantom{\big|} 
  \right|_{#2} 
  }}
\newcommand{\myrightarrow}[1]{\xrightarrow{\makebox[2em][c]{$\scriptstyle#1$}}}

\renewcommand{\phi}{\varphi}
\DeclarePairedDelimiter{\ceil}{\lceil}{\rceil}
\DeclarePairedDelimiter\floor{\lfloor}{\rfloor}
\newcommand{\var}[1]{\textrm{Var}\left[#1\right]}
\renewcommand{\varpi}[1]{\textrm{Var}_\pi\left[#1\right]}
\newcommand{\tmix}{\textrm{t}_                                                                                                                    \textsc{mix}(\epsilon)}

\newcommand{\infset}[1]{\inf \left\{ #1  \right\}}
\newcommand{\dtv}[2]{\textrm{d}_\textsc{tv}\left( #1, #2 \right)}

\newcommand{\wrt}{w.r.t }

\newcommand{\logn}{\log n}

\newcommand{\goodset}{\{\phi^{\theta_2}\leq L_2 + 4\}}
\newcommand{\good}{\textrm{Good}}

\newcommand{\Tbad}{T_{bad}}

\newcommand{\absolutevalue}[1]{\left|#1\right|}

\newcommand{\phiitheta}{\phi_i^{\theta}}
\newcommand{\phiithetatwo}{\phi_i^{\theta_2}}
\newcommand{\phitheta}{\phi^{\theta}}
\newcommand{\phithetatwo}{\phi^{\theta_2}}
\newcommand{\orange}[1]{\textcolor{black}{#1}}

\begin{document}

\title{\LARGE The mean-field Zero-Range process with unbounded monotone rates: mixing time, cutoff, and Poincaré constant}
 
\author{\Large Hong-Quan Tran\thanks{tran@ceremade.dauphine.fr}}
\affil{\large CEREMADE, CNRS, Université Paris-Dauphine, PSL University 75016 Paris, France}
\maketitle

\begin{abstract}
    We consider the mean-field Zero-Range process in the regime where the potential function $r$ is increasing to infinity at sublinear speed, and the density of particles is bounded. We determine the mixing time of the system, and establish cutoff. We also prove that the Poincaré constant is bounded away from zero and infinity. This mean-field estimate extends to arbitrary geometries via a comparison argument. Our proof uses the path-coupling method of Bubley and Dyer and stochastic calculus.    
\end{abstract}
\tableofcontents

\section{Introduction}
\subsection{Model} 
The Zero-Range process, introduced by Spitzer, is a model of interacting particle systems in continuous time. It describes the evolution of $m \geq 1$ indistinguishable particles jumping randomly across $n \geq 1$ sites, where the speed of a particle only depends on the number of its cooccupants (hence the name Zero-Range). More precisely, the interaction is represented by a function $r: \{1, 2, ...\} \to (0, \infty)$, called \textit{the potential function}, where $r(k)$ is the rate at which a site with $k$ particles expels a particle. For convenience, we let $r(0) = 0$ (no jump from empty sites). In this paper, we focus on the \textit{mean-field} version of the model, where a jumping particle chooses its destination uniformly among all sites. More precisely, we consider a continuous-time Markov chain $X := (X(t))_{t\geq 0} = (X_1(t), X_2(t),...,X_n(t))_{t\geq 0}$ taking values in the state space 
\begin{equation}
    \Omega := \left\{x = (x_1, x_2, ..., x_n) \in \Zbb_+^n: \sum_{i=1}^n x_i = m\right\}, 
\end{equation}
whose Markov generator $\Lcal$ acts on an observable $\phi : \Omega \to \Rbb$ as follows: 
\begin{equation}\label{generator}
    (\Lcal \phi)(x) = \dfrac{1}{n} \sum_{1\leq i, j \leq n} r(x_i)(\phi(x - \delta_i + \delta_j) - \phi(x)).
\end{equation}
Here  $(\delta_i)_{1\leq i \leq n}$ denotes the canonical basis of $\Zbb_+^n$. 
The generator $\Lcal$ is irreducible and reversible with respect to the following law: 
\begin{equation}\label{measurestationnaire}
    \pi(x) \propto \prod_{i = 1}^n \prod_{k = 1}^{x_i} \dfrac{1}{r(k)}, 
\end{equation}
with the convention that an empty product is 1. The classical theory of Markov processes says that starting from any probability on the state space $\Omega$, $X(t)$ will converge in distribution to the stationary law $\pi$ as $t\to \infty$. The speed of convergence from the state $x \in \Omega$ is quantified by the so-called mixing times: 
\begin{equation}
    \tmixxepsilon := \min\{t \geq 0: \dtv{P_x^t}{\pi} \leq \epsilon\}.
\end{equation}
Here $\dtv{\cdot}{\cdot}$ is the total variation distance, defined by $\dtv{\mu}{\nu} = \max\limits_{A\subset \Omega}|\mu(A) - \nu(A)| $, and $P_x^t$ denotes the distribution of $X(t)$ under the probability $\Pbb_x$: $P_x^t(\cdot) = \Pbb_x\left[X(t)\in \cdot \right]$, where $\Pbb_x$ is the law of the process starting from $x$. Of particular interest is the worst-case mixing time: 
\begin{equation}
    \tmix := \max\{\tmixxepsilon: x\in \Omega\},
\end{equation}
where we take the maximum of the mixing times over all initial configurations $x$.\\
We recall that the \textit{Dirichlet form} associated with our process is defined by: 
\begin{equation*}
    \Ecal(\phi,\psi) := -\crochet{\phi}{\Lcal \psi}_\pi,
\end{equation*}
where $\crochet{\phi}{\psi}_\pi := \sum\limits_{x \in \Omega} \pi(x)\phi(x)\psi(x)$ denotes the usual inner-product in $L^2(\Omega, \pi)$. Then the\textit{ Poincaré constant}, denoted by $\lambda_*$, is defined by: 
\begin{equation*}
    \lambda_* := \min \left\{\dfrac{\Ecal(\phi,\phi)}{\variance{\phi}}\right\},
\end{equation*}
where the minimum is taken over all non-constant observables, and $\variance{\phi}$ denotes the variance of $\phi$ under $\pi$. In our case, $\lambda_*$ coincides with the more classical absolute spectral gap due to reversibility of the system: 
\begin{equation*}
    \lambda_* = \lim_{t \to \infty} -\dfrac{1}{t}\log \max_{x\in \Omega} \dtv{P_x^t}{\pi}.
\end{equation*}
The purpose of the present paper is to estimate $\lambda_*$ and $\tmixxepsilon$, under certain assumptions on $r(\cdot)$.
\paragraph{Previous works.} To the best of our knowledge, the total-variation mixing time of the Zero-Range process has only been studied in a few cases: the case where $r$ is constant in \cite{lacoin2016cutoff}, \cite{lacoin2016mixing}, \cite{merle2018cutoff}, the case where $r$ is non-decreasing and bounded in \cite{hermon2020cutoff}, and the somehow-trivial case of independent walkers where $r$ is linear. Regarding the Poincaré constant, a notable result is given in $\cite{morris2006}$, where Morris determines the order of magnitude of $\lambda_*$ in the case where $r$ is constant
. Another result is obtained by Caputo in \cite{caputo2004spectral} for the case where $r$ is homogeneously Lipschitz and increasing at infinity, \ie 
\begin{align}
    \sup_{k \geq 1}|r(k+1) - r(k)| < \infty, \\
    \inf_{k - l \geq \delta} r(k) - r(l) > 0,
\end{align}
for some $\delta\in\Zbb_+$, where he proves that the Poincaré constant is bounded away from zero. In \cite{hermon2019version}, Salez and Hermon prove a comparison principle that allows us to compare the Poincaré constant of many models with that of the mean-field model. We will use this principle below.
\subsection{Main results}
We consider the ``intermediate'' regime where the function $r$ is non-decreasing, unbounded but grows slower than a linear function. More precisely, throughout the paper, we assume that $r$ satisfies:   
\begin{align}\label{rcondition1}
    &r(k+1) \geq r(k), \, \forall k \in \Zbb_+, \\
    \label{rcondition2}
    &\lim_{k\to\infty} r(k) = \infty,\\
    \label{rcondition3}
    &\sup_{k \in \Zbb_+} \dfrac{r(k)}{k} < \infty.
\end{align} 
We study the regime where the number of sites diverges while the density of particles per site remains bounded. More precisely, we always suppose that $m = m(n)$ and $x = x^{(n)}$, and all asymptotic statements refer to the regime:
\begin{equation}\label{densitybounded}
    n \to \infty, \qquad \dfrac{m}{n} \leq \rho,
\end{equation}
where $\rho$ is a positive constant. To lighten the notation, we keep the dependency upon $n$ implicit as much as possible. By a \textit{dimension-free} constant, we mean a number that depends only on $r$ and $\rho$. Our notation $\Ocal{\cdot}$ (resp. $\Omega(\cdot), \, \Theta(\cdot), \, o(\cdot)$) means being upper bounded by (resp. lower bounded by, upper and lower bounded by, negligible compared to) the quantity inside the brackets up to a dimension-free prefactor. 
We define a function $R: \{1,2,...\} \to \Rbb$ as follows: 
\begin{equation}
    \forall k \in \Zbb_+, \, R(k) = \sum_{i = 1}^k \dfrac{1}{r(i)}.
\end{equation}
\orange{Under condition \eqref{rcondition3}, we easily see  that
\begin{align}\label{Rdiverge}
    \lim_{k\to\infty} R(k) = \infty.
\end{align}}
We prove that $R(\norminfty{x})$ is a good estimate for $\tmixxepsilon$, as stated in the following theorem: 
\begin{theorem}[Main result]\label{mixing_time}
   For $\epsilon \in (0, 1)$ fixed, for any initial state $x$,
    \begin{equation}
        \tmixxepsilon \leq (1 + o(1))R(\norminfty{x}) + \Ocal{\logn}.
    \end{equation}
    In addition, if the initial state $x = x^{(n)}$ satisfies $\norminfty{x^{(n)}} \myrightarrow{n \to \infty} \infty$, then
    \begin{equation*}
        \tmixxepsilon \geq (1 - o(1))R(\norminfty{x}).
    \end{equation*}
\end{theorem}
Maximizing over all initial states $x$, we obtain
\begin{corollary}[Cutoff]\label{cutoff}
    Suppose additionally that $R(m) \gg \log n$. Then for $\epsilon \in (0, 1)$ fixed,
    \begin{equation}
        \dfrac{\tmix}{R(m)} = 1 +o(1).
    \end{equation}
    In other words, the system exhibits \textit{cutoff} at time $R(m)$.
\end{corollary}
The class of functions $r$ that satisfy conditions \orange{\eqref{rcondition1}, \eqref{rcondition2}, \eqref{rcondition3}} is quite large. A natural example is when $r$ is of the form $r(k) = k^\alpha, \, \forall k \in \Zbb_+$, for some $\alpha \in (0,1)$. In this case, 
\[R(k) = (1 + o(1))\dfrac{k^{1-\alpha}}{1-\alpha}, \, \text{as}\: k \to \infty,\]
by the Stolz-Cesàro Theorem. Thereupon, a direct application of our result gives the following.
\begin{example}
    Suppose that $r(k) = k^\alpha, \, \forall k \in \Zbb$, for some $\alpha \in (0,1)$, and suppose that $m \gg (\logn)^{1/(1-\alpha)}$. Then the system exhibits cutoff at time $\dfrac{m^{1-\alpha}}{1-\alpha}$.
\end{example}

Cutoff for the Zero-Range process was obtained in \cite{merle2018cutoff} for the case $r(k) = 1$ and more generally in \cite{hermon2020cutoff} for the case where $r$ is non-decreasing and bounded. Our work complements these results by investigating the case where $r \to \infty$. 
We also prove that the Poincaré constant is bounded away from zero and infinity:
\begin{theorem}[Poincaré constant]\label{spectral_gap}
        $\lambda_* = \Theta(1).$
\end{theorem}

Thanks to the comparisons in the paper \cite{hermon2019version} of Hermon and Salez, we can extend this result to the more general case where a jumping particle chooses its destination according to a doubly stochastic matrix $P$ rather than uniformly among all sites (for example, take $P$ to be the transition matrix of random walk on a regular graph). More precisely, let $P$ be an irreducible doubly stochastic transition matrix on $[n]:= \{1, 2, ..., n\}$, and let $\Lcal^P$ be the generator on $\Omega$ that acts on an observable $\phi: \Omega \to \Rbb$ by: 
\begin{equation*}
    (\Lcal^P \phi)(x) =  \sum_{1 \leq i, j \leq n} r(x_i)P(i,j)(\phi(x - \delta_i + \delta_j) - \phi(x)).
\end{equation*}
Similarly, we can define the Poincaré constants $\lambda_*(P)$ and $\lambda_*(\Lcal^P)$  of $P$ and $\Lcal^P$ via their associated Dirichlet forms and their stationary laws. Then we have the following. 
\begin{corollary}[Poincaré constant in arbitrary geometry]\label{application_spectral_gap}
   $\lambda_*(\Lcal^P) = \Theta(\lambda_*(P))$.
\end{corollary}
We give an example where we can compute $\lambda_*(P)$ explicitly to obtain explicit estimate on $\lambda_*(\Lcal^P)$. 
\begin{example}[Poincaré constant of torus model]\label{torus_example}
    Let $P$ be the transition matrix of the simple random walk on the lattice $\Zbb^d/p\Zbb^d$, for some $p, d \in \Zbb_+$. Then $\lambda_*(\Lcal^P) = \Theta(1/(dp^2))$, as $ p^d \to \infty$.
\end{example}
The calculation of $\lambda_*(P)$ is defered to the end of the paper. 
\paragraph{Heuristics.}
If we ignore arrivals and only consider departures of particles, then $R(\norminfty{x})$ is exactly the expectation of the time it takes for the initially highest site to be emptied. In the true system, due to the conditions imposed on $r$, the arrival rate at each site is uniformly bounded, and consequently, $R(\norminfty{x})$ remains a good approximation for the emptying-time. For the lower bound, we prove that before time $R(\norminfty{x})$, the initially highest site still has too many particles, and hence the system has not yet reached equilibrium. 
For the upper bound, we will see that at time $t = (1 + o(1))R(\norminfty{x}) + \Ocal{\logn} $, $\norminfty{X(t)} = \Ocal{ \log n}$. Afterwards, the system quickly reaches equilibrium.  
\paragraph{Acknowledgment.} The author warmly thanks Justin Salez for constructive discussions and his comments on the draft. \orange{The author also kindly thanks the anonymous referee for his suggestion to make the paper more clear and readable.}
\section{Lower bound on the mixing time}
\subsection{Preliminaries}\label{preliminaries}
We will use the following two graphical constructions of the process $X$.
\paragraph{Graphical construction 1.} Let $\Xi$ be a Poisson point process of intensity $\dfrac{1}{n} \textrm{d} t \otimes \textrm{d} u \otimes \textrm{Card} \otimes \textrm{Card}$ on $[0, \infty) \times [0, \infty) \times [n] \times [n]$, where $\textrm{Card}$ denotes the counting measure. Define the piece-wise constant process $X = (X(t))_{t\geq 0}$ taking values in $\Omega$ as follows: $X(0) = x$, and for each point $(t,u,i,j)$ of $\Xi$, 
\begin{equation}
        X(t) := \begin{cases}
        X(t-) - \delta_i + \delta_j, & \text{if } u \leq r(X_i(t-)) \\
        X(t-) &\text{otherwise}.\\
        \end{cases}
    \end{equation}
Then $X$ is a càdlàg Markov process starting from $x$ with generator $\Lcal$.
\paragraph{Graphical construction 2.}Let $\Psi$ be a Poisson point process of intensity $ \textrm{d}t \otimes \textrm{d}u \otimes \textrm{Card}$ on $[0, \infty) \times [0,\infty) \times [n]$. Consider the piece-wise constant process $X = (X(t))_{t\geq 0}$ which starts at $X(0) = x$ and has the following jumps: for each \orange{point} $(t,u,j)$ \orange{of} $\Psi$, 
    \begin{equation}
        X(t) := \begin{cases}
        X(t-) - \delta_i + \delta_j, & \text{if } \dfrac{1}{n} \sum\limits_{k = 1}^{i-1} r\left(X_k(t-)\right) < u \leq \dfrac{1}{n} \sum\limits_{k = 1}^i r(X_k(t-)),\orange{\text{ for some $i \in [n]$}}  \\
        X(t-) &\text{otherwise}.\\
        \end{cases}
    \end{equation}
Then $X$ is also a Markov process starting from $x$ with generator $\Lcal$. We can view the Poisson process in the graphical construction 2 as the repartition of the Poisson process in the graphical construction 1 according to the destination of the jumps.
\paragraph{Filtration.}We always note $(\Fcal_t)_{t\geq 0}$ \textit{the filtration} generated by the Poisson processes in the graphical construction we are using, where $\Fcal_t$ is the $\sigma-algebra$ generated by these processes up to time $t$. It is immediate that the process $X$ is adapted to the filtration and has càdlàg trajectories. 
\paragraph{Mean-field jump rate.}At any time $t$, as the model is mean-field, the arrival rate at each site is the same. We denote this quantity by $\zeta(t)$:
\begin{equation*}
    \zeta(t) := \dfrac{1}{n} \sum_{j = 1}^n r(X_j(t)).
\end{equation*}
Condition \eqref{rcondition3} implies that 
\begin{equation}\label{mf_jump_rate}
    \zeta(t) \leq \left(\dfrac{1}{n}\sum_{j=1}^n X_j(t)\right)\sup_{k \in \Zbb_+} \dfrac{r(k)}{k}\leq \rho\sup_{k \in \Zbb_+} \dfrac{r(k)}{k} =: \kappa.
\end{equation}
Hence, the number of particles arriving at each site is stochastically dominated by a Poisson process of dimension-free intensity $\kappa$.
\paragraph{Martingale associated with an observable.} For any observable $\phi: \Omega \to \Rbb$, under $\Pbb_x$, the process $M = (M(t))_{t \geq 0}$ given by 
\begin{equation}\label{martingale_formula}
    M(t) :=\phi(X(t)) - \phi(x) - \int_0^t\Lcal \phi(X(u)) \textrm{d} u 
\end{equation}
is a zero-mean martingale, see e.g \cite{EK}. Let $\phi_1,\, \phi_2$ be two observables, and let $M_1,\, M_2$ be the associated martingales. Then the \textit{predictable covariation} of $M_1$ and $M_2$ is given by 
\begin{equation}\label{covariation_formula}
    \crochet{M_1}{M_2}_t = \int_0^t\sum_{y\in \Omega} \Lcal(X(u), y)\left(\phi_1(y) - \phi_1(X(u))\right)\left(\phi_2(y) - \phi_2(X(u))\right)du.
\end{equation}
We recall a lemma on the concentration of martingales with jumps (see \cite{vandegeer1995}):
\begin{lemma}[Concentration of martingale]\label{martingale_concentration_lemma}
    Let $(M(t))_{t \geq 0}$ be a zero-mean càdlàg martingale w.r.t a filtration that satisfies the usual conditions. Suppose that $M(0) = 0$ and $M(t) - M(t-) \leq K$ for all $t > 0$ and some $0 \leq K < \infty$. Then for each $a >0, b>0$, 
    \begin{align}
        \proba{\exists t \geq 0:  M(t) \geq a, \crochet{M}{M}_t \leq b^2} \leq \exp{\left[-\dfrac{a^2}{2(aK + b^2)}\right]}.
    \end{align}
\end{lemma}

\paragraph{Gain/loss at a site.}We will need the following quantities:
\begin{enumerate}
    \item For $i \in [n]$, let $G_i(t)$ be the counting process that counts the number of particles arriving at site $i$ up to time $t$. We call $G_i$ \textit{the gain} at site $i$.
    \item Let $L_i(t)$ be the counting process that counts the number of particles jumping out of site $i$ up to time $t$. We call $L_i$ \textit{the loss} at site $i$. 
\end{enumerate}
Obviously, $X_i(t) = X_i(0) +G_i(t) - L_i(t)$.

\subsection{Elementary concentration inequalities}

We write $\xi \sim \exp(\lambda)$ to mean that $\xi$ is an exponential variable with parameter $\lambda$, \ie $\xi$ has density $\lambda e^{-\lambda x} \indicator{x > 0} \drm x$. We list here some useful inequalities, whose proofs are simple applications of Chernoff's bound (see, \eg \cite{boucheron:hal-00942704} for more details on Chernoff's bound).


\begin{lemma}[Poisson concentration]\label{concentrationPoisson}
    Let $Z$ be a Poisson variable. Then for any $B > 1$, 
        \[\proba{Z \geq B\Ebb Z} \leq e^{-(1+B\ln B - B)\Ebb Z}.\]
\end{lemma}

\begin{lemma}[Concentration of sum of independent exponential variables]
    Let $\lambda_1, ..., \lambda_k$ be positive numbers, and let $\xi_1, ..., \xi_k$ be independent random variables such that $ \xi_i \sim \exp(\lambda_i),\, \forall 1 \leq i \leq k $. Let 
    \begin{equation*}
        S := \sum_{i = 1}^k \xi_i.
    \end{equation*}
    For $B>0$ arbitrary, we have the following inequalities:
    \begin{equation}\label{concentration_inf}
        \proba{S - \Ebb S \leq -\sqrt{\var{S}}B} \leq e^{-B^2/4},
    \end{equation}
    and 
\begin{equation}\label{concentration}
    \proba{S - \Ebb S \geq \lambda \var{S} + \dfrac{B}{\lambda}} \leq e^{-B/2},
\end{equation}
where $\lambda = \min\limits_{1 \leq i \leq k}\{\lambda_i\}$.
\end{lemma}

\subsection{Proof of the lower bound}
In this subsection, we prove the lower bound on $\tmixxepsilon$ in Theorem \ref{mixing_time}.
First, we analyze the law of a single site at equilibrium.
\begin{proposition}[Single site marginals at equilibrium]\label{geometricdecay}
    There exists a dimension-free constant $q > 0$ such that 
    \begin{equation*}
        \forall k \in \Zbb_+, \quad \dfrac{\pi(x_1 = k)}{\pi(x_1 = k-1)} < \dfrac{q}{r(k)}.
    \end{equation*}
\end{proposition}

\begin{proof}
    \orange{Let $x$ be an arbitrary configuration in $\Omega$ such that $x_1 \geq 1$.} \orange{Thanks to \eqref{densitybounded},} the number of sites that have at least $2\rho$ particles is at most $\dfrac{n}{2}$, and hence the number of sites that have less than $2\rho$ particles is at least $ \dfrac{n}{2}$. For any $l \in [n]\setminus \{1\}$ such that $x_l < 2\rho$, \orange{thanks to \eqref{rcondition1}}, $$ \dfrac{\pi(x)}{\pi(x- \delta_1 + \delta_l)} = \dfrac{r(x_l+1)}{r(x_1)} \leq \dfrac{r(\ceil{2\rho})}{r(x_1)}.$$
    Taking average over all sites $l$ such that $x_l < 2\rho$, we obtain 
    \begin{align*}
        \pi(x) &\leq \dfrac{r(\ceil{2\rho})}{r(x_1)}\dfrac{2}{n} \orange{\sum_{\substack{l \neq 1\\x_l < 2\rho }}}  \pi(x-\delta_1+\delta_l)\\
        &\leq \dfrac{r(\ceil{2\rho})}{r(x_1)}\dfrac{2}{n} \sum_{l \neq 1} \pi(x-\delta_1+\delta_l).
    \end{align*}
    Now we take the sum over all $x$ such that $x_1 = k$ to get:
    \begin{align*}
        \pi(x_1 = k) &\leq \dfrac{r(\ceil{2\rho})}{r(k)}\dfrac{2}{n} \sum_{x_1 = k}\sum_{l \neq 1} \pi(x-\delta_1+\delta_l)\\
        &= \dfrac{r(\ceil{2\rho})}{r(k)}\dfrac{2}{n} \sum_{l \neq 1}\sum_{x_1 = k} \pi(x-\delta_1+\delta_l)\\
        &= \dfrac{r(\ceil{2\rho})}{r(k)}\dfrac{2}{n} \sum_{l \neq 1}\sum_{\substack{z_1 = k - 1, \\ z_l \geq 1}} \pi(z)  &\text{(change of variable: $z = x - \delta_1 + \delta_l)$}\\
        &\leq \dfrac{r(\ceil{2\rho})}{r(k)}\dfrac{2}{n} \sum_{l \neq 1}\sum_{\substack{z_1 = k - 1}} \pi(z)\\
        &\orange{=}  \dfrac{r(\ceil{2\rho})}{r(k)}\dfrac{2}{n} \sum_{l \neq 1}\pi(x_1 = k - 1)\\
        &\leq \dfrac{r(\ceil{2\rho})}{r(k)}\dfrac{2n}{n} \pi(x_1 = k-1)\\
        &= \dfrac{q}{r(k)}\pi(x_1 = k - 1),
    \end{align*}
    where $q = 2r(\ceil{2\rho}).$
\end{proof}
Now we study the effect of arrivals at a particular site. Recall that $L_i$ denotes the loss at site $i$, as defined at the end of subsection \ref{preliminaries}.
\begin{lemma}[Effect of arrivals]\label{timedescend}
    \orange{Let $x$ be an arbitrary initial configuration. }For $i \in [n]$ and $h \in \Zbb_+$ such that $x_i \geq h$, there exist independent variables $U_k \sim \exp\left(\dfrac{n-1}{n}r(k)\right), \, h \geq k \geq 1$, such that for any $0 \leq k \leq h-1$, 
    \begin{enumerate}
        \item $T_k := U_h + U_{h-1} + ... + U_{h-k}$ is a stopping time,
        \item \orange{almost surely,} $L_i(T_k) \geq k + 1$ ,
        \item \orange{almost surely,} $X_i(t) \geq h-k, \: \forall t \in [0, T_k)$.
    \end{enumerate}
\end{lemma}
The intuition is as follows: if we ignore arrivals and consider only departures at site $i$, then $T_k$ is the time to have $k+1$ particles depart from $i$. Arrivals, on the one hand, slow down a site from being emptied, but on the other hand, accelerate the rate of expelling and hence increase the loss. So the inequalities at points 2 and 3 follow.
\begin{proof}
    We use the graphical construction 1. 
    We first prove for the case $k = 0$. Let $U_h$ be defined by 
    \begin{equation*}
        U_h = \infset{t \geq 0: \Xi\Big([0,t] \times[0,r(h)] \times \{i\} \times [n]\setminus\{i\}\Big) > 0}.
    \end{equation*} It is clear that $U_h$ is a stopping time and $U_h \sim \exp\left(\dfrac{n-1}{n}r(h)\right)$.
    Moreover, by definition of $U_h$,
    $$\Xi\Big([0, U_h) \times [0,r(h)] \times \{i\} \times [n]\setminus\{i\}\Big) = 0,$$ 
    so before time $U_h$, $X_i$ cannot fall from $h$ to $h-1$. In other words, $\forall t \in [0, U_h), \: X_i(t) \geq h$. In particular, $X_i(U_h - )\geq h$, so there should be a jump from site $i$ to $[n]\setminus\{i\}$ at $U_h$. Consequently, $L_i(U_h) \geq 1$ and $X_i(U_h) = X_i(U_h-) - 1 \geq h-1$, \orange{almost surely}, which finishes the case $k = 0$. \\
    Now we define $U_k$ inductively by:
    \[U_k = \infset{t \geq 0: \Xi\Big((T_{h-k-1}, T_{h-k-1}+ t] \times [0,r(k)] \times \{i\} \times [n]\setminus\{i\}\Big) > 0},\]
    where $T_{h-k-1} = U_h + ... + U_{k+1}$\orange{.}
    The variables \orange{$(U_k)_{h \geq k \geq 1}$} are independent by the stationary and independent increments of Poisson processes. In addition, 
    \[U_k \sim \exp\left(\dfrac{n-1}{n} r(k)\right).\]
    The claim is simply obtained by induction and by the strong Markov property.
\end{proof}
\paragraph{Useful variables.} Lemma \ref{timedescend} allows us to compare certain random times with the random variables $(S_k)_{k \geq 1}$ defined by 
\begin{equation}\label{defsumofexponentialvariable}
    S_k = \sum\limits_{i = 1}^k \xi_i, 
\end{equation}
where $(\xi_i)_{i \geq 1}$ is a sequence of independent random variables such that $\xi_i\sim\exp(r(i)), \, \forall i \in \Zbb_+$.
Obviously, $\esperance{S_k} = \sum\limits_{i=1}^k \dfrac{1}{r(i)} = R(k)$, $\var{S_k} = \sum\limits_{i=1}^k \dfrac{1}{r(i)^2}$. \orange{Due to \eqref{rcondition1}, \eqref{rcondition2}, \eqref{rcondition3},} the functions $r, R$ diverge, \orange{so} we easily see that 
\begin{align}
    \label{ratioexpectationvar}
    &\lim\limits_{k\to \infty}\dfrac{\var{S_k}}{\Ebb S_k} = 0.
\end{align}
\paragraph{} Lemma \ref{timedescend} makes precise the fact that arrivals can only slow down a site from being emptied, while Proposition \ref{geometricdecay} \orange{together with condition \eqref{rcondition2}} say that at equilibrium, the typical height of a site cannot be very large. This leads to the lower bound on $\tmixxepsilon$ in Theorem \ref{mixing_time}:
\begin{proof}[Proof of the lower bound in Theorem \ref{mixing_time}] 
    Let $\delta \in (0,1)$ be fixed, and let $x \in \Omega$ be arbitrary. We only need to prove that for $\norminfty{x}$ sufficiently large,
    \begin{equation}
    \dfrac{\tmixxepsilon}{R(\norminfty{x})} > 1 - \delta.
    \end{equation} 
    Without loss of generality, suppose that site $1$ is originally the highest, \ie $\norminfty{x} = x_1$. We know that for any $A \subset \Omega$, 
    \[\dtv{P^t_x}{\pi}\geq P^t(x, A) - \pi(A).\]
    We choose $$A = \{y \in \Omega: y_1 \geq k\},$$ where 
    \begin{equation*}
        k = \sup\left\{ l \in \Zbb_+: |R(l)| \leq\dfrac{\delta}{2} R(\norminfty{x})\right\}.    
    \end{equation*}
    We only need to show that $\pi(A) = o(1)$, and for $t = (1 - \delta)R(x_1)$, $P^t_x(A) = 1 - o(1)$.\\
    \orange{Thanks to \eqref{Rdiverge}}, $\norminfty{x} \gg 1$ ensures that \orange{$R(\norminfty{x}) \gg 1$ and hence} $k \gg 1$, so by Proposition \ref{geometricdecay} \orange{and \eqref{rcondition2}}, $\pi(A) = o(1)$. On the other hand, we apply Lemma \ref{timedescend} with $i = 1, h = x_1$ to conclude that there exists a stopping time  
    \[T_{x_1 - k - 1} \overset{(d)}{=} \dfrac{n}{n-1}(S_{x_1} - S_k),\]
    where the sequence $(S_k)_{k \geq 1}$ is defined in \eqref{defsumofexponentialvariable}, such that $X_1(t) \geq k + 1, \: \forall t \in[0, T_{x_1 - k -1})$. We \orange{define} $S_{k, x_1} = S_{x_1} - S_k$, \orange{for any $k < x_1$}. Then 
    \begin{align*}
        P_x^{(1-\delta)R(x_1)}\left[A\right] &\geq \probawithstartingpoint{x}{T_{x_1 - k - 1} > (1-\delta)R(x_1)} \\&\geq \proba{S_{k,x_1}> (1-\delta)R(x_1)} \\ 
        &= 1 - \proba{S_{k,x_1} \leq (1-\delta)R(x_1)}.
    \end{align*}
    Moreover, $R(k) \leq \dfrac{\delta}{2}R(x_1)$ by definition of $k$, and hence $\esperance{S_{k, x_1}} = R(x_1) - R(k) \geq (1-\delta/2)R(x_1)$. So by the concentration inequality \eqref{concentration_inf},
    \begin{align*}
        \proba{S_{k,x_1} \leq (1-\delta)R(x_1)} &\leq \proba{S_{k,x_1} - \Ebb S_{k,x_1} \leq -\dfrac{\delta}{2}R(x_1)}\\
        &\leq \exp\left(-\dfrac{1}{4}\cdot\dfrac{\delta^2}{4}R(x_1)^2\var{S_{k,x_1}}^{-1}\right)\\ 
        &\leq \exp\left(-\dfrac{\delta^2}{16}R(x_1)\var{S_{x_1}}^{-1}\right).\\
        &=o(1)\orange{,}
    \end{align*}
    \orange{where in the third inequality we have used the fact that $\variance{S_{x_1}} > \variance{S_{k, x_1}}$, and in the last equality we have used \eqref{ratioexpectationvar} and the fact that $x_1 = \norminfty{x} \myrightarrow{n \to \infty} \infty$.}
    Hence $P^t_x(A) = 1 - o(1)$, which finishes our proof.
\end{proof}

\section{Upper bound on the mixing time }
We will prove the following statements:
\begin{proposition}[Dissolution]\label{key_proposition}
    There exist dimension-free constants $\sigma, \alpha_1$ such that for \orange{any} $\delta \in (0,1)$ fixed, for any $x \in \Omega$, for any $t \geq (1 + \delta)R(\norminfty{x}) + \sigma \logn$,
    \begin{align*}
        \probawithstartingpoint{x}{\norminfty{X(t)} \geq \alpha_1 \logn} = \Ocal{n^{-2}}.
    \end{align*}
\end{proposition}    
\begin{proposition}[Quick convergence to equilibrium]\label{quick_convergence}
    Let $\alpha_1$ be defined as in Proposition \ref{key_proposition}. Then there exists a dimension-free constant $\alpha$ so that for any configuration $x$ such that $\norminfty{x} \leq \alpha_1 \logn$,
    \begin{align*}
        \dtv{P^{\alpha \logn}_x}{\pi} = \Ocal{n^{-2}}.
    \end{align*}
\end{proposition}

First we see how these propositions lead to the upper bound on $\tmixxepsilon$ in Theorem \ref{mixing_time}:
\begin{proof}[Proof of the upper bound in Theorem \ref{mixing_time}:]
    \orange{Let $\sigma, \alpha_1, \alpha$ be defined as in Proposition \ref{key_proposition} and Proposition \ref{quick_convergence}. }Let $x$ be an arbitrary configuration; let $t_1 = (1 + \delta)R(\norminfty{x}) + \sigma\logn$, \orange{for some $\delta$}, $t_2 = \alpha \logn$, and $t = t_1 + t_2$. We only need to prove that for \orange{arbitrary $\delta \in (0,1)$ fixed}, for $n$ large enough, for any $x \in \Omega$, 
    \begin{equation*}
        \tmixxepsilon \leq t.
    \end{equation*}
    By the convexity of the total variation distance, 
    \begin{align*}
        \dtv{P^t_x}{\pi} &\leq \sum_{y \in \Omega} P^{t_1}_x(y) \dtv{P^{t_2}_y}{\pi}\\
        &\leq \probawithstartingpoint{x}{\norminfty{X(t_1)} \geq \alpha_1 \logn} + \max_{\norminfty{y} \leq \alpha_1 \logn } \dtv{P^{t_2}_y}{\pi},
    \end{align*}
    which is $\Ocal{n^{-2}}$ by Proposition \ref{key_proposition} and Proposition \ref{quick_convergence}, hence smaller than $\epsilon$ when $n$ is large enough.
\end{proof}
\orange{The structure of the rest of this section is as follows. In Subsection \ref{Dissolution}, we prove Proposition \ref{key_proposition} and provide some analysis on the trajectory of the system which will be used in the proof of Proposition \ref{quick_convergence}. In Subsection \ref{path_coupling}, we prove Proposition \ref{quick_convergence} by the path coupling method of Bubley and Dyer. }

\subsection{Dissolution}\label{Dissolution}
Recall that $G_i$ denotes the gain at site $i$, as defined at the end of subsection \ref{preliminaries}. First we give an estimate on $G_i$ at the predicted time.
\begin{lemma}[Estimating the gain at a site]\label{bound_gain} 
        For any dimension-free constant $d$, there exists a dimension-free constant $c_0$ such that at time $t = (1 + \delta/4)(R(\norminfty{x}) + d\logn)$, \orange{for any $\delta \in (0,1)$ fixed, } 
        \begin{equation}\label{smallgain}
            \proba{G_i(t) \geq  c_0(R(\norminfty{x})\vee \logn)} = \Ocal{n^{-5}} .
        \end{equation}
\end{lemma}

\begin{proof}
    We use the graphical construction 2. Since $\zeta(t) < \kappa$ at all time (see \eqref{mf_jump_rate}), $G_i$ is simply dominated by the Poisson process $\Psi\big|_{\cdot \times[0, \kappa] \times \{i\}}$. Consequently, at time $t = (1+\delta/4)(R(\norminfty{x})+ d\logn)$, $G_i(t)$ is stochastically dominated by a random variable $Y \sim $ $\textrm{Poisson}\Big(\kappa (1+\delta/4) (R(\norminfty{x})+ d\logn)\Big)$. Then the result is simply obtained by Lemma \ref{concentrationPoisson}, for $c_0$ large enough.
\end{proof}
For any site $i$, Lemma \ref{timedescend} says that arrivals can only accelerate the loss $L_i$, while Lemma \ref{bound_gain} gives us a (random) upper bound on $G_i$. They together lead to the following proposition:
\begin{proposition}[First phase dissolution]\label{first_phase_dissolution}
    Let $d = \dfrac{7}{r(1)}$, and let $c_0$ be defined as in Lemma \ref{bound_gain}. Then for any $x \in \Omega$, \orange{for any $\delta \in (0,1)$ fixed,} for $t = (1+ \delta/4)(R(\norminfty{x}) + d\logn)$, 
    \begin{equation}
        \probawithstartingpoint{x}{\norminfty{X(t)} \leq c_0(R(\norminfty{x})\vee \logn)} = 1 - \Ocal{n^{-2}}.    
    \end{equation}
\end{proposition}
\begin{proof} 
    Let $i \in [n]$. We apply Lemma \ref{timedescend} with $h = x_i$ to conclude that there exists a stopping time $T_{x_i - 1} \overset{(d)}{=} \dfrac{n}{n-1} S_{x_i}$ such that $L_i(T_{x_i - 1}) \geq x_i$. Note that $S_{x_i}$ is dominated stochastically by $S_{\norminfty{x}}$, hence by using \eqref{concentration}, we deduce that
    \begin{align*}
        &\proba{T_{x_i - 1} \geq \dfrac{n}{n-1}\left(\esperance{S_{\norminfty{x}}} + r(1) \var{S_{\norminfty{x}}} + \dfrac{6\logn}{r(1)}\right)}\\ &= \proba{S_{x_i} \geq \left(\esperance{S_{\norminfty{x}}} + r(1) \var{S_{\norminfty{x}}} + \dfrac{6\logn}{r(1)}\right)}\\
        &\leq \proba{S_{\norminfty{x}} \geq \left(\esperance{S_{\norminfty{x}}} + r(1) \var{S_{\norminfty{x}}} + \dfrac{6\logn}{r(1)}\right)} \\
        &\leq n^{-3}.
    \end{align*}
    Note that $\var{S_{\norminfty{x}}} = o(R(\norminfty{x})\vee \logn)$ \orange{by \eqref{ratioexpectationvar}}, and hence 
    \[\dfrac{n}{n-1}\left(\esperance{S_{\norminfty{x}}} + r(1) \var{S_{\norminfty{x}}} + \dfrac{6\logn}{r(1)}\right) < t,\]
    when $n$ is large enough. Consequently, for $n$ large enough, 
    \[\probawithstartingpoint{x}{L_i(t) < x_i} \leq \probawithstartingpoint{x}{T_{x_i-1} > t} \leq n^{-3}.\]
    We take a union bound of this and the event in \eqref{smallgain} over all sites to conclude that  $$\proba{\exists i, \, G_i(t) \geq c_0(R(\norminfty{x})\vee \logn)} + \proba{\exists i, \, L_i(t) < x_i} = \Ocal{n^{-2}}.$$ 
    The claim follows.
\end{proof}
We now recall a simple version of Gronwall's lemma that we will use a lot:
\begin{lemma}[Gronwall's lemma]\label{differentialinequality}
    Let $\alpha, \beta$ be some positive numbers. Let $u: [0, \infty) \to \Rbb^+$ be a continuously differentiable function such that $\ddt u(t) \leq -\beta u(t) + \alpha$. Then 
    \[u(t) < \dfrac{\alpha}{\beta} + \left(u(0) - \dfrac{\alpha}{\beta}\right)e^{-\beta t}.\]
    In particular, if $t \geq \dfrac{\log u(0)}{\beta}$, then $u(t) < \dfrac{\alpha}{\beta} + 1$.
\end{lemma}
For $\theta$ a positive number, and for $i \in [n]$, we define the observable $\phiitheta: \, \Omega \to \Rbb $ by 
\begin{align}
    \phiitheta(x) = e^{\theta x_i},
\end{align}
and we define the observable $\phitheta$ by 
\begin{align*}
    \phitheta(x) = \dfrac{1}{n} \sum_{i = 1}^n \phiitheta(x).
\end{align*}

\begin{lemma}[Estimate on $\Lcal \phi^\theta$]\label{boundL}
    For any dimension-free constants $\theta, \beta > 0$, there exists a number $L = L(\theta, \beta)$  such that, \orange{for any configuration $x$,}
    \begin{align}
        &\Lcal \phi^\theta(x) \leq -\beta \phi^\theta(x) \indicator{\phi^\theta(x) >L} + (e^\theta - 1)\kappa \phi^\theta(x) \indicator{\phi^\theta(x) \leq L}\label{boundLphitheta}.
    \end{align}
    In particular, 
    \begin{align}\label{boundLphitheta2}
        \Lcal \phi^\theta(x) \leq -\beta \phi^\theta(x) + ((e^{\theta} - 1)\kappa + \beta) L.
    \end{align}
\end{lemma}

\begin{proof}
        For simplicity, we write $\phi$ instead of $\phi^\theta$ and $\phi_i$ instead of $\phi_i^\theta$.
         It is not difficult to see that \begin{align*}
            \dfrac{\Lcal \phi_i(x)}{\phi_i(x)} &=  \dfrac{e^\theta - 1}{n}\sum\limits_{j \in [n] \setminus \{i\}}r(x_j) - \dfrac{1 - e^{-\theta}}{n}\sum\limits_{j \in [n] \setminus \{i\}}r(x_i)\\
            &= \dfrac{e^\theta - 1}{n}\sum\limits_{j \in [n] }r(x_j) - \left((1- e^{-\theta})\dfrac{n-1}{n} + \dfrac{e^\theta - 1}{n} \right)r(x_i).\\
        \end{align*} 
        Hence\orange{, by \eqref{mf_jump_rate}}, \begin{align*}
            \Lcal \phi_i(x) &\leq (e^\theta - 1)\kappa\phi_i(x) - (1 - e^{-\theta})\phi_i(x) r(x_i).
        \end{align*}
        Taking the average over all sites $i$ we get 
        \begin{align*}
            \Lcal \phi(x) \leq (e^\theta - 1) \kappa \phi(x) - (1 - e^{-\theta}) \dfrac{\sum_{i \in [n]} r(x_i) \phi_i(x)}{n}.
        \end{align*}
        The claim follows when $\phi(x) \leq L$. It remains to consider the case $\phi(x) > L$. For any $c \in \Zbb_+$, $r(x_i) \phi_i(x) \geq r(c)(\phi_i(x) - e^{\theta c})$ due to the monotonicity of $r$, hence 
        \begin{align*}
            \sum_{i \in [n]} r(x_i) \phi_i(x) &\geq r(c) \sum_{i \in [n]} (\phi_i(x) - e^{\theta c})\\
            &\geq r(c)(n \phi(x) - n e^{\theta c}).
        \end{align*}
        Consequently, 
        \begin{align*}
            \Lcal \phi(x) \leq (e^\theta - 1) \kappa \phi(x) - (1 - e^{-\theta}) r(c)(\phi(x) - e^{\theta c}).
        \end{align*}
        Let $L = L(c) = 2e^{\theta c}$. If $\phi(x) > L$, then $\phi(x) - e^{\theta c} > \dfrac{\phi(x)}{2}$, which implies:
        \begin{align}
            \Lcal \phi(x) \leq (e^\theta - 1) \kappa \phi(x) - (1 - e^{-\theta}) r(c) \dfrac{\phi(x)}{2}.
        \end{align}
        We can take $c$ large enough to make the right-hand side of the inequality above smaller than $-\beta \phi(x)$, which finishes the proof \orange{of \eqref{boundLphitheta}}. \orange{\eqref{boundLphitheta2} is obtained by rewriting \eqref{boundLphitheta} as follows, 
        \begin{align*}
            \Lcal \phi(x) &\leq -\beta \phi(x) (1 - \indicator{\phi(x) \leq L}) + (e^\theta - 1)\kappa \phi(x) \indicator{\phi(x) \leq L}\\
            &= -\beta \phi(x) + ((e^\theta - 1)\kappa + \beta)\phi(x)\indicator{\phi(x) \leq L}\\
            &\leq -\beta \phi(x) + ((e^\theta - 1)\kappa + \beta)L,
        \end{align*}
        which is what we want.}
\end{proof}
A good estimate on $\Lcal \phi^\theta$ will guarantee good behavior of the trajectories of $X$, as stated in the following proposition.
\begin{proposition}[Dissolution]\label{concentration_and_drift} Let $\theta, \beta > 0$ be some dimension-free constants. Then for any $x \in \Omega$, for any $t\geq\dfrac{\theta}{\beta} \norminfty{x}$, 
        \begin{align}\label{logngood}
            \probawithstartingpoint{x}{\norminfty{X(t)} \geq \dfrac{6}{\theta} \logn} = \Ocal{ n^{-5}}.
        \end{align}
\end{proposition}

\begin{proof}
    We still write $\phi$ instead of $\phi^\theta$, for simplicity. Let $L = L(\theta, \beta)$ be defined as in Lemma \ref{boundL}.
        Let $u(t) = \esperancewithstartingpoint{x}{\phi(X(t))}$. By \eqref{martingale_formula}, 
        \begin{align*}
            \ddt u(t) &= \esperancewithstartingpoint{x}{\Lcal \phi(X(t))}.
        \end{align*}
        This and \eqref{boundLphitheta2} imply:
        \begin{align*}
            \ddt u(t) \leq -\beta u(t) + \orange{((e^\theta - 1)\kappa + \beta)}L.
        \end{align*}
        Therefore, by Lemma \ref{differentialinequality},
        \begin{align*}
            u(t) \leq \dfrac{\orange{((e^\theta - 1)\kappa + \beta)}L}{\beta} + u(0) e^{-\beta t}.
        \end{align*}
        Hence for $t \geq \dfrac{\theta}{\beta} \norminfty{x}$, $u(t) \leq \dfrac{\orange{((e^\theta - 1)\kappa + \beta)}L}{\beta} + 1$.  
        Note that $e^{\theta \norminfty{x}} \leq n\phi(x)$, hence $\esperance{e^{\theta \norminfty{X(t)}}} \leq nu(t)$. Then the claim is a simple consequence of Chernoff's bound.
\end{proof} 
\deleted{Now we fix a number $\theta_1$, and we choose $\beta_1$ large enough such that $\dfrac{\theta_1 c_0}{\beta_1} \leq \dfrac{\delta}{2}$. Let $\alpha_1 = \dfrac{6}{\theta_1}$ and $L_1 = L(\theta_1, \beta_1)$ as in Lemma \ref{boundL}.} We \deleted{can} now prove Proposition \ref{key_proposition}:
\begin{proof}[Proof of Proposition \ref{key_proposition}]
    \orange{We fix $\delta \in (0,1)$. Let $c_0$ be defined as in Lemma \ref{bound_gain}. Let $\theta_1 > 0$ be fixed, and let $\beta_1$ be such that $\dfrac{\theta_1 c_0}{\beta_1} \leq \dfrac{\delta}{2}$. Let $L_1 = L(\theta_1, \beta_1)$ as in Lemma \ref{boundL}.} Let $t_1 = (1 + \delta/4)(R(\norminfty{x}) + d\logn), \; t_2 = \dfrac{\delta}{2}(R(\norminfty{x})\vee \logn)$. By definition of $\beta_1$ and Proposition \ref{first_phase_dissolution},
    \[\probawithstartingpoint{x}{\dfrac{\theta_1\norminfty{X(t_1)}}{\beta_1} \leq t_2} = 1 - \Ocal{n^{-2}}.\]
    \orange{By Markov property at time $t_1$ and inequality \eqref{logngood} 
    , we conclude that for any $t \geq t_1 + t_2,$
    \[\probawithstartingpoint{x}{\norminfty{X(t)} \geq \dfrac{6}{\theta_1}\logn} = \Ocal{n^{-2}}.\]
    We choose $\alpha_1 = \dfrac{6}{\theta_1}$ and $\sigma = \dfrac{5}{4}d + \dfrac{1}{2}$ to conclude the proof.}
\end{proof}
The estimate on $\Lcal \phitheta$ in Lemma \ref{boundL} also ensures that the system quickly reaches the set where $\phitheta$ is small.
\begin{proposition}[Exponential moment of hitting time]\label{moment_hitting_time}
    Let $\theta, \beta > 0$ be some dimension-free constants, and let $L = L(\theta, \beta)$ be defined as in Lemma \ref{boundL}.
    Let $T$ be the hitting time of the set $\{x \in \Omega: \phi^\theta(x) \leq L\}$. Then for any $x \in \Omega$,
        \begin{align}\label{tempsdattentelaplace}
            \esperancewithstartingpoint{x}{e^{\beta T}} \leq e^\theta \phi^\theta(x) /L.
        \end{align}
\end{proposition}
\begin{proof}
        We write $\phi$ instead of $\phitheta$.
        The idea is that if $\phi$ is large, then the drift $\Lcal \phi$ is negative, and its magnitude is of the same order as $\phi$. Hence the system will quickly reach the set where $\phi$ is small, which will be made precise by stochastic calculus. Consider the function $F: \Rbb^2 \to \orange{\Rbb} $ defined by $F(u,v) = uv$, which is twice continuously differentiable. By \eqref{martingale_formula},
        \[\phi(X(t)) = \phi(X(0)) + \int_0^t \Lcal \phi(X(u)) \textrm{d} u + M(t),\]
        where \orange{$M(t)$} is a martingale. Moreover, $\phi(X)$ is a pure-jump process since $X$ is piece-wise constant. For a càdlàg process $Y$, we denote by $\Delta Y$ its jumps: $\Delta Y(s) = Y(s) - Y(s-)$. We \deleted{note} \orange{define} $G(t) = F(e^{\beta t}, \phi(X(t)))$. Note that the function $t \mapsto e^{\beta t}$ has bounded variation. Consequently, applying Itô's formula (for example, see \orange{Theorem 33 in chapter 2 of} \cite{Pro2005}) to the function $F$ and the semi-martingales $t \mapsto \phi(X(t))$ and  $t \mapsto e^{\beta t}$, we get: 
        \begin{equation}\label{formule_ito}
            \begin{split}
             G(t) = &\phi(X(0)) + \int_0^t e^{\beta u } \Lcal \phi(X(u)) \textrm{d} u + \int_0^t e^{\beta u} \drm M(u) + \int_0^t \phi(X(u)) \beta e^{\beta u} \drm u\\
                &+ \sum_{0 \leq s \leq t} \left[ \Delta G(s) - \dfrac{\partial F}{\partial v}\left(e^{\beta s-}, \phi(X(s-))\right) \cdot \Delta \phi(X(s))  \right].
            \end{split}
        \end{equation}
        On the other hand, as $\dfrac{\partial F}{\partial v}(u,v) = u$ and the function $t \mapsto e^{\beta t}$ is continuous, 
        \[\Delta G(s) = e^{\beta s} \phi(X(s)) - e^{\beta s} \phi(X(s-)) = e^{\beta s} \Delta \phi(X(s)) =  \dfrac{\partial F}{\partial v}\left(e^{\beta s-}, \phi(X(s-))\right) \cdot \Delta \phi(X(s)). \]
        Hence the last term in the right-hand side of \eqref{formule_ito} is zero. Moreover, the term $\int_0^t e^{\beta u} \drm M(u)$ is a martingale. Applying the formula at time $t \wedge T$, we get: 
        \begin{equation}
        \begin{split}
            &e^{\beta (t\wedge T)} \phi(X(t\wedge T)) \\
            &= \phi(X(0)) + \int\limits_0^{t\wedge T} e^{\beta u } \Big(\Lcal \phi(X(u)) + \beta \phi(X(u))\Big) \drm u  + \int\limits_0^{t\wedge T} e^{\beta u} \drm M(u).
        \end{split}
        \end{equation}
        \deleted{Since} \orange{By Lemma \ref{boundL},} $\Lcal \phi(X(u)) + \beta \phi(X(u)) \leq 0$ when $u < T$\orange{.} It follows that $e^{\beta (t\wedge T)} \phi(X(t\wedge T))$ is a supermartingale. Thus, 
        \begin{align}\label{exponential_moment_phi_stop}
            \esperancewithstartingpoint{x}{e^{\beta (t\wedge T)} \phi(X(t\wedge T))} \leq \phi(x).
        \end{align}
        Clearly, if $x \in \arg \min \phi$, then $\Lcal \phi(x) \geq 0$, and hence by \eqref{boundLphitheta}, $\phi(x) \leq L$. In particular, $\{\phi \leq L\} \neq \varnothing$, so $T < \infty$ a.s. as the process is irreducible. In \eqref{exponential_moment_phi_stop}, letting $t \to \infty$ and using Fatou's lemma, we obtain
        \begin{align*}
            \esperancewithstartingpoint{x}{e^{\beta  T} \phi(X(T))} \leq \phi(x).
        \end{align*}
        It is easy to see that $\phi(X(T)) \geq \phi(X(T-))/e^\theta \geq Le^{-\theta}$, which leads to our claim.
\end{proof}
The next lemma says that if we start from a configuration $x$ such that $\norminfty{x} = \Ocal{\logn}$, then this remains true for a long time.
\begin{lemma}[Stability of trajectories]\label{stabilitylogngood}
    \orange{Let $\alpha_1$ be defined as in Proposition \ref{key_proposition}. }There is a dimension-free constant $\alpha_2$ such that  \begin{align}
        \sup\limits_{\norminfty{x} \leq \alpha_1 \logn}\probawithstartingpoint{x}{\exists t \in [0, (\logn)^2], \norminfty{X(t)} > \alpha_2\logn}  = \Ocal{n^{-3}}.
    \end{align}
\end{lemma}
\begin{proof}
   Suppose that $\norminfty{x} \leq \alpha_1 \logn$. \orange{Let $\theta_1 = \dfrac{6}{\alpha_1}$, and let $\beta_1$ be a constant,} and let $t = \dfrac{\theta_1}{\beta_1}\alpha_1\log n \geq \dfrac{\theta_1}{\beta_1} \norminfty{x}$. Then by \eqref{logngood}, 
    \[\probawithstartingpoint{x}{\norminfty{X(t)} > \alpha_1 \logn} = \Ocal{ n^{-5}}.\]
    Moreover, by Lemma \ref{concentrationPoisson}, for a dimension-free constant $\alpha_2'$ large enough,
    \[\proba{G_i(t) \geq \alpha_2' \log n} = \Ocal{n^{-5}}.\]
    Taking a union bound, we deduce that 
    \begin{align*}
        \probawithstartingpoint{x}{\{\norminfty{X(t)} > \alpha_1\logn\} \cup \{\exists i : G_i(t) \geq \alpha_2' \log n \} } = \Ocal{n^{-4}}.
    \end{align*}
    This implies
    \begin{align*}
        \probawithstartingpoint{x}{\norminfty{X(t)} \leq \alpha_1\logn , \sup_{s \in [0, t]} \norminfty{X(s)} \leq (\alpha_1 + \alpha_2') \logn} \geq 1 - \Ocal{n^{-4}}.
    \end{align*}
    The inequality remains true when we take the supremum over all $x$ such that $\norminfty{x} \leq \alpha_1 \logn$. Iterating, and using the Markov property, we deduce that\orange{, for any $k\in \Zbb_+$, }
    \begin{align*}
        \probawithstartingpoint{x}{\norminfty{X(kt)} \leq \alpha_1 \logn,\sup_{s \in [0, kt]} \norminfty{X(s)}\leq (\alpha_1 + \alpha_2')\logn } \geq ( 1 - \Ocal{n^{-4}})^k \geq 1 - \Ocal{kn^{-4}}.
    \end{align*}
    We finish the proof simply by taking \orange{$k = \floor{(\logn)^2}$}, and $\alpha_2 = \alpha_1 + \alpha_2'$.
\end{proof}
In the next proposition, we prove that the bound on $\norminfty{X}$ above leads to a strong bound on $\phi^{\theta}(X)$ for some $\theta > 0$. 
\begin{proposition}[Strong concentration of trajectories]\label{stability}
   \deleted{Let $\theta_2$ be a constant small enough such that $\theta_2 \leq \theta_1$ and $(\logn)^2 n^{2\theta_2\alpha_2 - 1} = \Ocal{\dfrac{1}{\sqrt{n}}}$. }Let \orange{$\theta_2$ and }$\beta_2$ be \orange{two} positive \orange{dimension-free} constants and $L_2 = L(\theta_2, \beta_2)$ as in Lemma \ref{boundL}. Let $T$ be the hitting time of the set $\{\phithetatwo \leq L_2\}$. \orange{Let $\alpha_1$ be defined as in Proposition \ref{key_proposition}}. Then, \orange{provided that $\theta_2$ is small enough,} for any $x$ such that $\norminfty{x} \leq \alpha_1 \logn$,  
    \begin{align}
        \probawithstartingpoint{x}{\sup_{s\in [T, (\logn)^2]} \phithetatwo(X(s)) > L_2+4} = \Ocal{n^{-3}}.
    \end{align}
\end{proposition}
We will need the following lemma:
\begin{lemma}[Martingale estimate]\label{lemma_for_prop_stability}
    \orange{Let $\alpha_1$ be defined as in Proposition \ref{key_proposition}. Let $\theta_2$ be a positive dimension free constant. }Suppose that $\norminfty{x} \leq \alpha_1 \logn$, and let $(M(t))_{t\geq 0}$ be defined by 
    \begin{equation*}
        M(t) = \phithetatwo(X(t)) - \phithetatwo(X(0)) - \int_0^t \Lcal \phithetatwo(X(u)) \drm u,    
    \end{equation*}
    \orange{which is a martingale according to \eqref{martingale_formula}.} Then, \orange{for $\theta_2$ small enough,}  
    \[\probawithstartingpoint{x}{\sup_{s \in [0, (\logn)^2]} |M(s)| \geq 1} = \Ocal{n^{-3}}.\]
\end{lemma}
For simplicity, in the proofs of Lemma \ref{lemma_for_prop_stability} and Proposition \ref{stability}, we still write $\phi$ instead of $\phithetatwo$ and $\phi_i$ instead of $\phiithetatwo$. First we see how Lemma \ref{lemma_for_prop_stability} leads to Proposition \ref{stability}:
\begin{proof}[Proof of Proposition \ref{stability}]
    \orange{Let $\theta_2$ and $M$ be as in Lemma \ref{lemma_for_prop_stability}. }We will prove that
    \[\left\{\sup_{s \in [T, (\logn)^2]}\phi(X(s)) > L_2 + 4\right\} \bigsubset[1.7] \left\{\sup_{s \in [0, (\logn)^2]} |M(s)| \geq 1\right\} \bigcup\, \big\{T \geq (\logn)^2\big\},\]
    and then we show that the probabilities of the events on the right-hand side is $\Ocal{n^{-3}}$. By contrapositivity, suppose that for a realization of $X$ \orange{which is càdlàg almost surely}, we have \orange{$\sup\limits_{s \in [0, (\logn)^2]}|M(s)| < 1$} and $ T < (\logn)^2$. We prove that 
    \[\sup_{s \in [T, (\logn)^2]}\phi(X(s)) \leq L_2 + 4.\]
    For $h \in [T, (\logn)^2]$ arbitrary, let $s_h = \sup\{s \in [0,h]: \phi(X(s-)) \leq L_2\}$. Note that for any $s \in [0, (\logn)^2]$,  \[ |\Delta\phi(X(s))| = |\Delta M(s)| \leq 2\sup_{s \in [0, (\logn)^2]} |M(s)| \leq 2.\] Moreover, by definition of $s_h$, $\phi(X(s_h-)) \leq L_2$, and hence $\phi(X(s_h)) \leq L_2 + \Delta\phi(X(s_h)) \leq L_2 + 2$. Also by definition of $s_h$, $\phi(X(u)) > L_2$ when $s_h \leq u < h$, and hence $\Lcal \phi(X(u)) < 0$ \orange{by Lemma \ref{boundL}}. Consequently,
    \begin{align}
        \phi(X(h)) = \phi(X(s_h)) - M(s_h) + M(h) + \int_{s_h}^h \Lcal \phi(X(u)) \drm u \leq L_2+4,
    \end{align}
    which proves the inclusion. Besides, Proposition \ref{moment_hitting_time} gives us $\probawithstartingpoint{x}{T \geq (\logn)^2} = \Ocal{n^{-3}}$ by Chernoff's bound. Combining this and Lemma \ref{lemma_for_prop_stability}, we deduce the claim.
\end{proof}
Now we prove Lemma \ref{lemma_for_prop_stability}:
\begin{proof}[Proof of Lemma \ref{lemma_for_prop_stability}]
    We will provide good \orange{control} on $\Delta M$ and $\crochet{M}{M}$, and afterward we use Lemma \ref{martingale_concentration_lemma}.
    By \eqref{martingale_formula}, the process $M_i(t)$ defined by
    \begin{align*}
        M_i(t) = \phi_i(X(t)) - \phi_i(X(0)) - \int_0^t \Lcal \phi_i(X(u)) \drm u
    \end{align*}
    is a zero-mean martingale. It it clear that $M$ is the average of $M_i$:
    \begin{align*}
        M(t) = \dfrac{1}{n} \sum_{i = 1}^n M_i(t).
    \end{align*}
    Due to the conservation of the number of particles, the martingales $(M_i)_{i \in [n]}$ have negative covariations. More precisely\orange{, according to \eqref{covariation_formula}}, we have
    \begin{align*}
        d\crochet{M_i}{M_j}_t = \sum_{1 \leq k, l \leq n} \dfrac{r(X_k(t))}{n}\Big(\phi_i(X(t) - \delta_k + \delta_l) - \phi_i(X(t))\Big)\Big(\phi_j(X(t) - \delta_k + \delta_l) - \phi_j(X(t))\Big).
    \end{align*}
    Note that when $i \neq j$, $\left(\phi_i(x - \delta_k + \delta_l) - \phi_i(x)\right)\left(\phi_j(x - \delta_k + \delta_l) - \phi_j(x)\right)$ is negative if $\{k, l\} = \{i, j\}$ and is zero otherwise. Hence $\crochet{M_i}{M_j}_t \leq 0, \, \forall i\neq j$. Consequently, for all $t$ positive,  
    \begin{align*}
        \crochet{M}{M}_t \leq \dfrac{1}{n^2} \sum_{i = 1}^n \crochet{M_i}{M_i}_t.
    \end{align*}
    \orange{Let $\alpha_2$ be defined as in Lemma \ref{stabilitylogngood}, } and let $U$ be the exit time from $\{\norminfty{\cdot} \leq \alpha_2 \logn\}$. \orange{Note that, almost surely, for any $u \geq 0$, $\Delta M(u) = 0$ or there exist $k, l \in [n]$ such that 
    \[\Delta M(u) = \phi(X(u) - \delta_k + \delta_l) - \phi(X(u)) = \dfrac{1}{n} (e^{\theta_2} - 1)(e^{\theta_2X_l(u-)} - e^{\theta_2 (X_k(u-) - 1)}).\]} \deleted{Note that before time $U$,}\orange{In either case, before time $U$, almost surely,
    \begin{align*}
        |\Delta M(u)| \leq \dfrac{1}{n} (e^{\theta_2} - 1) e^{\theta_2 \norminfty{X(u-)}} \leq (e^{\theta_2} - 1) n^{\theta_2 \alpha_2 - 1}.
    \end{align*}        }   
    \orange{Similarly, 
    \[|\phi_i(x - \delta_k + \delta_l) -\phi_i(x)| \leq (e^{\theta_2} - 1) e^{\theta_2 \norminfty{x}}(\indicator{k = i} \vee \indicator{l = i}).\]
    Hence for any $u < U$, 
    \[|\phi_i(x - \delta_k + \delta_l) -\phi_i(x)|^2\leq (e^{\theta_2} - 1)^2 n^{2\theta_2 \alpha_2}(\indicator{k = i}\vee \indicator{l = i}).\]
    Then by dividing the double sum to the sum where $k = i$ or $l = i$ or $k\neq i \neq l$, we get
    }
    \begin{align*}
        \crochet{M_i^U}{M_i^U}_t &= \int_0^{t \wedge U} \sum_{1 \leq k, l \leq n} \dfrac{r(X_k(u))}{n}\left(\phi_i(X(u) - \delta_k + \delta_l) - \phi_i(X(u))\right)^2 \drm u \\
        &\orange{\leq \int_0^{t \wedge U} (r(X_i(u)) + \zeta(u))(e^{\theta_2} - 1)^2 n^{2\theta_2 \alpha_2} \drm u },
    \end{align*}
    \deleted{Hence for $t = (\logn)^2$,} \orange{where we recall that $\zeta(u)$ is the mean-field jump rate. Taking the sum over $i \in [n]$, we get }
    \begin{align*}
        \crochet{M^U}{M^U}_t &\leq \orange{\dfrac{1}{n^2} \int_0^{t \wedge U} 2 n \zeta(u)\, n^{2 \theta_2\alpha_2 }(e^{\theta_2} - 1)^2 \drm u} \\
        &\leq \dfrac{1}{n^2} \int_0^{t \wedge U} 2 n \kappa\, n^{2 \theta_2\alpha_2 }(e^{\theta_2} - 1)^2 \drm u\\
        &= \Ocal{tn^{2\theta_2 \alpha_2 - 1}},
    \end{align*}
    \orange{where we have used \eqref{mf_jump_rate} in the second inequality. Now for $\theta_2$ small enough, $2\theta_2 \alpha_2 - 1 < -1/2$, which implies $\Delta M(u)= \Ocal{1/\sqrt{n}}$ and $\crochet{M^U}{M^U}_t = \Ocal{1/\sqrt{n}}$ if $u < U$ and $t = (\logn)^2$.} Then we apply Lemma \ref{martingale_concentration_lemma} to the martingale $M^{U \wedge (\logn)^2}(\cdot) := M(U \wedge (\logn)^2 \wedge \cdot)$, with $a = 1$, $K = b^2 = \Ocal{1/\sqrt{n}}$, to obtain 
    \begin{align*}
        \probawithstartingpoint{x}{\sup_{s \geq 0} M^{U \wedge (\logn)^2}(s) \geq 1 } \leq e^{-\Omega(\sqrt{n})}.
    \end{align*}
    Using the same argument for $-M$, and taking a union bound, we deduce that 
    \begin{align*}
        \probawithstartingpoint{x}{\sup_{s \geq 0} \absolutevalue{M^{U \wedge (\logn)^2}(s)} \geq 1 } \leq 2 e^{-\Omega(\sqrt{n})}.
    \end{align*}
    Taking a union bound with the event in Lemma \ref{stabilitylogngood}, we get 
    \begin{align*}
        \probawithstartingpoint{x}{\exists s \in [0, (\logn)^2] :|M_s| \geq 1} \leq \probawithstartingpoint{x}{\sup_{s \geq 0} \absolutevalue{M^{U \wedge (\logn)^2}(s)} \geq 1 } + \probawithstartingpoint{x}{U \leq (\logn)^2} = \Ocal{n^{-3}},
    \end{align*}
    which finishes our proof.
\end{proof}

\subsection{Path coupling via tagged particles}\label{path_coupling}
For $k\in \Zbb_+$, define 
\begin{equation}
    \Delta(k) := r(k+1) - r(k) \geq 0.
\end{equation}
Let $\Theta$ be a Poisson point process of intensity $\dfrac{1}{n}\drm t\otimes \drm u\otimes \textrm{Card}$ on $\Rbb_+ \times \Rbb_+ \times [n]$, independent of the Poisson processes used in the graphical construction of $X$. For a site $i \in [n]$, define an $[n]$- valued process $I = (I(t))_{t\geq 0}$ by setting $I(0) = i$ and for each $(t,u,k)$ in $\Theta$, 
\begin{equation}
    I(t) := \begin{cases} 
                        k &\text{if } u \leq \Delta(X_{I}(t-))\\
                        I(t-) &\text{otherwise},\\
            \end{cases}
\end{equation}
where $X_I(t) := X_{I(t)}(t)$. This definition means that conditionally on $X$, the process $I(t)$ will jump with the time-varying rate $\Delta(X_{I}(t))$, and the destination is uniformly chosen among all sites. A simple but important observation is that $(X(t) + \delta_{I(t)})_{t\geq 0}$ has the same distribution as a Zero-Range process starting at $x + \delta_i$. We call $I$ a tagged particle, and we stress here that the construction of $I$ relies strictly on condition \eqref{rcondition1}. For $j$ another site, similarly we can construct a second tagged particle $J$ starting from $J(0) = j$ using the same process $\Theta$. Thus we have a coupling $(X(t) + \delta_{I(t)}, X(t) + \delta_{J(t)})_{t\geq 0}$ of two Zero-Range processes starting at $x + \delta_i$ and $x+\delta_j$ respectively.\orange{We call the particles of $X$ \textit{non tagged particles}.} We note $\Pbb_{x, i, j}$ for the law of the process $(X, I, J)$ starting from $(x, i, j)$ and $\Ebb_{x, i, j}$ for the expectation taken \wrt $\Pbb_{x, i, j}$. Let $\tau$ be the coalescence time of $I$ and $J$: 
\begin{equation}
    \tau := \inf\{t \geq 0: I(t) = J(t)\}.
\end{equation}
By the classical relation between $\dtv{\cdot}{\cdot}$ and coupling, we have:
\begin{equation}\label{couplinginequality}
    \dtv{P^t_{x + \delta_i}}{ P^t_{x+\delta_j}} \leq \probawithstartingpoint{x, i, j}{\tau > t}.
\end{equation}
By construction, if the two tagged particles manage to jump at the same time, then coalescence occurs. However, the jump rates of the tagged particles depend on their number of cooccupants, which complicates our task. We will need to analyze carefully the trajectories of $(X, I, J)$ to obtain a good estimate on $\tau$.\\
\orange{The inequality \eqref{couplinginequality} gives us the estimate on the total variation distance of two processes starting from two adjacent configurations. However, comparing processes starting from two arbitrary configurations directly by coupling is in general not easy. Nevertheless, if we are interested only in total variation, and if the comparison using coupling of processes from adjacent configurations is simpler, we can extend the comparison to two arbitrary configurations by choosing a path between them and then use triangle inequality. This simple but powerful idea, originally due to Bubley and Dyer in \cite{bubley1997path}, is the strategy that we will implement.}\\
\orange{Throughout this subsection, $\alpha_1$ will be as in Proposition \ref{key_proposition}, and $\theta_2, \beta_2, L_2$ will be as in Proposition \ref{stability}.} Proposition \ref{stability} and \eqref{tempsdattentelaplace} imply that starting from any configuration $x$ such that $\norminfty{x} \leq \alpha_1 \logn$, the system will reach the set $\{\phithetatwo \leq L_2\}$ quickly (in time $\Ocal{\logn}$) and then remains in $\{\phithetatwo \leq L_2 + 4\}$ for a long time (namely $\Omega((\logn)^2)$). We will prove that if this is the case, then the coalescence time $\tau$ is likely to be $\Ocal{\logn}$.

From now on, let $\good := \goodset$. We say that a configuration $x$ is \textit{good} if $x \in \good$ and $x$ is \textit{bad} otherwise.  
We introduce the process $(X^*, I^*)$ taking values in $\good \times [n]$ whose infinitesimal generator $\Lcal ^*$ acts on an observable $\phi: \good \times [n] \to \Rbb$ by:
\begin{equation}
    \Lcal ^*\phi(x,i) = \sum_{k,l}\dfrac{r(x_k)}{n}\left(\phi(x  - \delta_k + \delta_l, i) - \phi(x,i)\right)  \indicator{x - \delta_k + \delta_l \in \good} + \sum_{j = 1}^n\dfrac{\Delta(x_i)}{n}(\phi(x,j) - \phi(x,i)).
\end{equation}
This definition means that $X^*$ is a Zero-Range process constrained to staying good, and $I^*$ jumps with a time-varying rate $\Delta(X^*_I(t)) := \Delta(X^*_{I^*(t)}(t))$ and chooses its destination uniformly among all sites. We can use the same Poisson processes in the graphical construction of $(X, I)$ and $(X^*, I^*)$ to obtain a coupling of them, and we can construct the second tagged particle $J^*$ analogously. By their construction, if the processes $(X, I, J)$ and $(X^*, I^*, J^*)$ start from the same configuration $(x,i, j)$, then they coincide up to the time when $X$ turns bad. 
For any $\theta > 0$, we define the observable $\phi^\theta_*: \good \times [n] \to \Rbb_+$ by 
\begin{equation*}
    \phi_*^\theta(x, i) := \phi^\theta_i(x) = e^{\theta x_i}.
\end{equation*}
\begin{lemma}[Exponential moment of the number of cooccupants]\label{exponential_moment_cooccupants} 
    There exist dimension-free constants $c_2$ and $K$ such that \orange{for any $(x, i, j) \in \good \times [n] \times [n]$,} when $t \geq c_2 (x_i \vee x_j)$,  
    \begin{align}\label{laplacecoocupant}
        \esperancewithstartingpoint{x,i,j}{e^{\theta_2(X^*_I(t) \vee (X^*_J(t))}} \leq K.
    \end{align}
\end{lemma}
\begin{proof}
    We will prove that there exist positive dimension-free constants $a_2, b_2$ such that
    \begin{equation}\label{boundL*phi}
        \Lcal ^*\phi_*^{\theta_2}(x,i) \leq -a_2\phi_*^{\theta_2}(x,i) + b_2,
    \end{equation}
    for any $(x,i) \in \good \times [n]$. Let us see how \eqref{boundL*phi} leads to the claim:\\
    Let $u(t) = \esperancewithstartingpoint{x, i}{\phi^{\theta_2}_*(X^*(t), I^*(t))}$. Then 
    \begin{equation*}
        u'(t) = \esperancewithstartingpoint{x, i}{\Lcal^* \phi^{\theta_2}_*(X^*(t), I^*(t))} \leq -a_2u(t) + b_2.
    \end{equation*}
    Hence, by Lemma \ref{differentialinequality},
    \begin{equation*}
        \esperancewithstartingpoint{x, i}{e^{\theta_2(X_I^*(t))}} \leq \dfrac{b_2}{a_2} + 1,
    \end{equation*}
    for any $t \geq \dfrac{\theta_2}{a_2}x_i$. We take $c_2 = \dfrac{\theta_2}{a_2}, \, K = 2\left(\dfrac{b_2}{a_2} + 1\right)$, so the claim follows from the inequality
    \[e^{\theta_2(X^*_{I}(t) \vee X^*_{J}(t))} \leq e^{ \theta_2X^*_{J}(t)} + e^{\theta_2X^*_{I}(t)}.\]
    It remains to prove \eqref{boundL*phi}.
    It is similar to Lemma \ref{boundL} except now we have an extra term corresponding to the jump of the tagged particle $I$, which is controlled by the fact that the system is constrained to staying good. More precisely,
    \begin{align*}
    \Lcal^*\phithetatwo_*(x,i) = &\sum_{k \neq i}\dfrac{r(x_k)}{n}\left(e^{\theta_2} - 1\right)\phithetatwo_*(x,i)\indicator{x - \delta_k + \delta_i \in \good} \\
    &+ \sum_{k \neq i} \dfrac{r(x_i)}{n}\left(e^{-\theta_2} - 1\right)\phithetatwo_*(x, i)  \indicator{x - \delta_i + \delta_k \in \good}\\
    &+ \sum_{j = 1}^n\dfrac{\Delta(x_i)}{n}(\phithetatwo_*(x,j) - \phithetatwo_*(x,i)).
\end{align*}
    We will bound the three terms above to obtain an upper bound on $\Lcal ^*\phithetatwo_*$:\\
    The first term: since \orange{$\dfrac{1}{n}\sum\limits_{k=1}^n r(x_k) \leq \kappa$ by \eqref{mf_jump_rate}}, hence,
    \begin{align*}
       \sum_{k \neq i}\dfrac{r(x_k)}{n}\left(e^{\theta_2} - 1\right)\phithetatwo_*(x,i)\indicator{x - \delta_k + \delta_i \in \good} \leq \kappa(e^{\theta_2} - 1) \phithetatwo_*(x,i).
    \end{align*}
    The third term is negative when $x_i >\dfrac{1}{\theta_2} \log (L_2+4)$ \orange{since $x \in \good$}, hence,
    \[\Delta(x_i)\left(\dfrac{\sum_{j = 1}^n e^{\theta_2 x_j}}{n} - e^{\theta_2 x_i}\right) < \max_{k \leq \dfrac{1}{\theta_2} \log (L_2+4)} \Delta(k)(L_2 + 4) \orange{=:} c,\]
    for some dimension-free constant $c$.\\
    For the second term, we first observe that there are at most $\dfrac{n}{N}$ sites $l$ such that $x_l > N\rho$, \orange{for any constant $N>0$, thanks to \eqref{densitybounded}}. On the other hand, if $x_i > N\rho$ and $x_l \leq N\rho$, then $\phithetatwo(x - \delta_i + \delta_l) \leq \phithetatwo(x)$, so if $x$ is good, then so is $x - \delta_i + \delta_l$. Consequently, as $e^{-\theta_2} - 1$ is negative, we have  
    \begin{align*}
        &\sum_{k\neq i} \dfrac{r(x_i)}{n} (e^{-\theta_2} - 1)\phi_*^{\theta_2}(x, i) \indicator{x - \delta_i + \delta_k \in \good}\\
        \leq &\sum_{k\neq i} \dfrac{r(x_i)}{n} (e^{-\theta_2} - 1)\phi_*^{\theta_2}(x, i) \indicator{x - \delta_i + \delta_k \in \good} \indicator{x_i > N\rho}\\
        \leq &\sum\limits_{k \neq i}\dfrac{\orange{r(\ceil{N\rho})}}{n}(e^{-\theta_2} - 1)\phi_*^{\theta_2}(x, i) \indicator{x_i > N\rho} \indicator{x_k \leq N\rho}\\
        \leq & \orange{r(\ceil{N\rho})}(e^{-\theta_2} - 1)\phi_*^{\theta_2}(x, i)\indicator{x_i > N\rho} (1 - 1/N)\\
        = &\orange{r(\ceil{N\rho})}(e^{-\theta_2} - 1)\phi_*^{\theta_2}(x, i) (1 - 1/N) + \orange{r(\ceil{N\rho})}(1 - e^{-\theta_2} )\phi_*^{\theta_2}(x, i) (1 - 1/N)\indicator{x_i \leq N\rho}.
    \end{align*}
    We sum the three inequalities, and \orange{afterwards} we take 
    \begin{equation*}
        a_2 = -\kappa(e^{\theta_2}- 1) + (1 - 1/N)\orange{r(\ceil{N\rho})}(1- e^{-\theta_2})    
    \end{equation*}
    and
    \begin{equation*}
        b_2 = c + \orange{r(\ceil{N\rho})}e^{\theta_2N \rho}. 
    \end{equation*}
    We choose $N$ large enough to make $a_2 > 0$, which is what we needed.
\end{proof}
\orange{We fix a constant $c_2$ which satisfies Lemma \ref{exponential_moment_cooccupants}.} For any initial configuration $(x, i, j) \in \good \times [n] \times [n]$, we define successively the stopping times $(T_k)_{k \geq 1}$ as follows:
\begin{equation}\label{definition_of_Tk}
\begin{split}
    T_1 &= c_2(x_i \vee x_j) + 1 = c_2(X^*_I(0) \vee X^*_J(0)) + 1,\\
    T_k &= T_{k - 1} + c_2\left(X^*_{I}(T_{k-1})\vee X^*_{J}(T_{k-1}) \right) + 1.
\end{split}
\end{equation}
\begin{lemma}[Bound of $\tau$ by $T_k$]\label{induction1}
    \orange{Let $(T_k)_{k\geq 1}$ be defined as in \eqref{definition_of_Tk}.} Then there exists a dimension-free constant $c_3$ such that for any $(x, i, j) \in \good \times [n] \times [n]$, \orange{for any $k\geq 1$,}
    \begin{equation}\label{quickcoalescencelowoccupant}
       \probawithstartingpoint{x,i,j}{\tau \geq T_k;\restr{X}{[0, T_k]} = \restr{X^*}{[0, T_k]}} \leq (1-c_3)^k.
    \end{equation}
\end{lemma}
\begin{proof}
    We only need to prove for $k = 1$, then use induction and the strong Markov property. \orange{Let $c_2$ be the constant used in the definition of $(T_k)_{k\geq 1}$, and let $K$ be the corresponding constant in Lemma \ref{exponential_moment_cooccupants},} and let $t = c_2(x_i \vee x_j) = T_1 - 1$.
    By \eqref{laplacecoocupant} and Chernoff's bound, 
    \begin{align*}
        \probawithstartingpoint{x,i,j}{X^*_I(t) \vee X^*_J(t) \geq a} \leq \dfrac{K}{e^{\theta_2a}},
    \end{align*}
    for any $a  > 0$. We choose $a$ large enough to make the right-hand side less than $1/2$. We will prove that for \orange{any} $(x, i, j) \in \good \times [n] \times [n]$ such that $x_i \vee x_j \orange{<} a$, there exists a dimension-free constant $c > 0$ such that  
    \begin{align}\label{conditional_quick_coalescence}
         \probawithstartingpoint{x, i, j}{\tau < 1} > c.
    \end{align}
    Assuming for the moment that we have \eqref{conditional_quick_coalescence}, let us prove the lemma. It is not hard to see that
    \begin{align*}
        & \probawithstartingpoint{x, i, j}{\tau > T_1, \restr{X}{[0, T_1]} = \restr{X^*}{[0, T_1]}} \\
        &\leq \probawithstartingpoint{x,i,j}{X^*_I(t) \vee X^*_J(t) \geq a}  + \probawithstartingpoint{x,i,j}{X^*_I(t) \vee X^*_J(t) < a, \restr{X}{[0, t]} = \restr{X^*}{[0, t]}, \tau \orange{>} t + 1} \\
        &\leq \probawithstartingpoint{x,i,j}{X^*_I(t) \vee X^*_J(t) \geq a}  + (1-c)\probawithstartingpoint{x,i,j}{X^*_I(t) \vee X^*_J(t) < a} \\
        &\leq 1 - c\probawithstartingpoint{x,i,j}{X^*_I(t) \vee X^*_J(t) < a}\\ &\leq 1 - c/2.
    \end{align*}
    In the second inequality, we have used \eqref{conditional_quick_coalescence} and the Markov property at time $t$. We deduce the claim simply by taking $c_3 = c/2$. It remains to prove \eqref{conditional_quick_coalescence}:\\ Suppose that $x$ is good and $x_i \vee x_j < a$. The scenario is that in a finite time, there is no particle arriving at $i$ and $j$, and the tagged particles wait for two sites $i, j$ to be completely emptied, and \orange{afterwards} they jump at the same time. More precisely, we use the mixed graphical construction for the process $X$ as follows: let $\Xi$ and $\Psi$ be two independent Poisson processes defined as in Graphical construction 1 and Graphical construction 2. Consider the process $X$ which starts from $x$ and has the following jumps: for each $(t, u, e) \in \Psi$ where $e \in \{i, j\}$,
     \begin{equation}
        X(t) := \begin{cases}
        X(t-) - \delta_l + \delta_e, & \text{if } \dfrac{1}{n} \sum\limits_{k = 1}^{l-1} r\left(X_k(t-)\right) < u \leq \dfrac{1}{n} \sum\limits_{k = 1}^l r(X_k(t-)),\orange{\text{ for some $l \in [n]$}}  \\
        X(t-) &\text{otherwise},\\
        \end{cases}
    \end{equation}
    and for each $(t, u, k, l) \in \Xi$ where $l \in [n] \setminus \{i, j\}$,
    
    \begin{equation}
        X(t) := \begin{cases}
        X(t-) - \delta_k + \delta_l, & \text{if } r(X_k(t-)) \geq u  \\
        X(t-) &\text{otherwise}.\\
        \end{cases}
    \end{equation}
    Then $X$ is a Markov process with generator $\Lcal$ on $\Omega$. Here we use $\Psi$ to indicate the jumps to two sites $i,j$ and $\Xi$ to indicate other jumps. Let
    \begin{itemize}
        \item $A = \{\Psi([0, 1] \times [0, \kappa] \times \{i, j\}) = 0\}$,
        \item $B_i = \{\Xi([0, 1/2] \times [0, r(1)]\times \{i\} \times [n]\setminus\{i, j\}) \geq a\}$,
        \item $B_j = \{\Xi([0, 1/2] \times [0, r(1)]\times \{j\} \times [n]\setminus\{i, j\}) \geq a\}$,
        \item $C = \{\Theta([0, 1/2] \times [0,\max\limits_{k \leq a} \Delta(k)] \times [n]) = 0)\} \cap \{\Theta([1/2, 1] \times [0, r(1)] \times [n]) \geq 1\}$.
    \end{itemize}
    \orange{In fact, $A$ is the event that there is no non-tagged particle arriving at two sites $i, j$ up to time $1$. If $A$ happens, then $B_i$ and $B_j$ ensure that all non-tagged particles of two sites $i, j$ jump to $[n] \setminus\{i,j\}$ in $[0, 1/2]$. If $A, B_i, B_j$ happen, then two sites $i, j$ are empty in $[1/2,1]$, and event $C$ ensures that the two tagged particles stay at $\{i,j\}$ in $[0, 1/2]$ then jump at the same time (hence coalescence) in $[1/2, 1]$. Moreover, the Poisson random variables used in the definitions of these events are independent and have parameters $\Theta(1)$.} \orange{We conclude that} the events above are independent and their probabilities are $\Theta(1)$. \deleted{Moreover, in their intersection, two tagged particles jump at the same time in $[1/2, 1]$, and hence}\orange{It follows that} \[\probawithstartingpoint{x, i, j}{\tau < 1} \geq \probawithstartingpoint{x, i, j}{A \cap B_i \cap B_j \cap C} = \Theta(1),\]
   which finishes our proof.
    
\end{proof}
\begin{lemma}[Exponential moment of $T_k$]\label{induction2}
\orange{Let $c_2$ and $(T_k)_{k \geq 1}$ be as in \eqref{definition_of_Tk}, and let $K$ be the corresponding constant in Lemma \ref{exponential_moment_cooccupants}.} Let $\theta_3 = \theta_2/c_2$. Then for any $(x, i, j) \in \good \times [n] \times [n]$, \orange{for any $k \geq 1$,}
   \begin{equation*}
        \esperancewithstartingpoint{x,i,j}{e^{\theta_3 T_k}} \leq n(Ke^{\theta_3})^k  (L_2+4).    
   \end{equation*}
\end{lemma}
\begin{proof}
    \orange{By convention, let $T_0 = 0$. For any $k \geq 2$,} note that $T_{k-1}$ is $\Fcal_{T_{k-2}}$-measurable by its definition. Conditionally on $\Fcal_{T_{k-2}}$, we have 
    \begin{align*}
        \esperancewithstartingpoint{x,i,j}{e^{\theta_3T_k} } &= \esperancewithstartingpoint{x,i,j}{\esperancewithstartingpoint{x,i,j}{e^{\theta_3T_k}|\Fcal_{T_{k-2}}}}\\
        &= \esperancewithstartingpoint{x,i,j}{e^{\theta_3 (T_{k-1} + 1)} \esperancewithstartingpoint{x,i,j}{e^{\theta_3c_2(X^*_I(T_{k-1}) \vee X^*_J(T_{k-1}))}|\Fcal_{T_{k-2}}}}\\
        &= \esperancewithstartingpoint{x,i,j}{e^{\theta_3 (T_{k - 1} + 1)} \esperancewithstartingpoint{X^*(T_{k - 2}),I^*(T_{k - 2}),J^*(T_{k - 2})}{e^{\theta_2(X^*_{I}(T_1)\vee X^*_{J}(T_1))  }   }} \\
        &\leq Ke^{\theta_3} \esperance{e^{\theta_3 T_{k - 1}} },
    \end{align*}
    where the inequality is due to \eqref{laplacecoocupant}. Moreover, 
    \begin{align*}
        \esperancewithstartingpoint{x, i, j}{e^{\theta_3 T_1}} = \esperancewithstartingpoint{x,i,j}{e^{\theta_2(x_i \vee x_j) + \theta_3 }} \leq n(L_2+4)e^{\theta_3},  
    \end{align*}
    \orange{where the last inequality is due to the fact that $x \in \good$}. The claim is then obtained by induction.
\end{proof}

\begin{corollary}[Quick coalescence while staying good]\label{staying_good_coalescence}
    There exists a dimension-free constant $\alpha_3$ such that for any $(x, i, j) \in \good \times [n] \times [n]$,
    \begin{align}\label{bound_tau_from_goodset}
         \probawithstartingpoint{x, i, j}{\tau > \alpha_3 \logn;X(t) = X^*(t), \, \forall t \leq \alpha_3 \logn} = \Ocal{n^{-3}}.
    \end{align}
\end{corollary}

\begin{proof}
    \orange{For any $\alpha_3 > 0$ and $k \in \Zbb_+$,} the left-hand side is upper bounded by  
    \begin{align*}
        & \probawithstartingpoint{x, i, j}{T_k > \alpha_3 \log n} + \probawithstartingpoint{x,i,j}{\tau \geq T_k;\restr{X}{[0, T_k]} = \restr{X^*}{[0, T_k]}}\\
        &\leq n(Ke^{\theta_3 })^k\orange{(L_2+4)} n^{-\theta_3 \alpha_3 } + (1-c_3)^k,
    \end{align*}
    \deleted{where the last inequality is} \orange{for some dimension-free constants $K, \theta_3, c_3$,} due to Lemma \ref{induction1}, Lemma \ref{induction2}, and Chernoff's bound.
    We choose $k = \Ocal{\logn}$ such that $(1 - c_3)^k = \Ocal{n^{-3}}$ and $\alpha_3$ large enough such that $(Ke^{\theta_3 })^k n^{1 - \theta_3 \alpha_3} = \Ocal{n^{-3}}$ to get what we wanted.
\end{proof}
Now we can finally prove the quick coalescence for a configuration $x$ such that $\norminfty{x} \leq \alpha_1\logn$.
\begin{proposition}[Quick coalescence]\label{coalescence_logn}
   \orange{Recall that $\alpha_1$ is fixed in this subsection.} There exists a dimension-free constant $\alpha$ such that for any $x$ such that $\norminfty{x} \leq \alpha_1\logn$, for any $i, j \in [n]$, 
    \begin{align}\label{quick_coalescence}
       \probawithstartingpoint{x, i, j}{\tau \geq \alpha \logn} = \Ocal{n^{-3}}.
   \end{align}
\end{proposition}
\begin{proof}
    Let $\alpha = \alpha_4 + \alpha_3$, where \orange{$\alpha_3$ is as in Corollary \ref{staying_good_coalescence}, and} $\alpha_4$ is a dimension-free constant that we will choose later. Let $T$ be the hitting time of the set $\{\phi^{\theta_2} \leq L_2\}$.
    The probability that we want to estimate does not exceed the following sum: 
    \begin{equation}
        \begin{split}
             &\probawithstartingpoint{x}{T \geq \alpha_4 \logn} \\
        &+ \probawithstartingpoint{x}{\sup_{s\in [T, (\logn)^2]} \phithetatwo(X(s)) > L_2+4}   \\
        &+ \probawithstartingpoint{x, i, j}{T < \alpha_4 \logn; \sup_{s\in [T; (\logn)^2]} \phithetatwo(X(s)) \leq L_2+4; \tau \geq T +  \alpha_3\logn }.
        \end{split}
    \end{equation}
    We simply prove that all the terms are $\Ocal{n^{-3}}$:
    \begin{enumerate}
        \item \textbf{The first term:} By \eqref{tempsdattentelaplace} and Chernoff's bound,   it is upper bounded by $\Ocal{n^{\theta_2 \alpha_1 - \beta_2\alpha_4}}$, which is $\Ocal{n^{-3}}$ when $\alpha_4$ is large enough. 
        \item \textbf{The second term} is $\Ocal{n^{-3}}$ by Proposition \ref{stability}.
        \item \textbf{The last term:} From the time $T$ onward, we couple the processes $(X,I, J)$ and $(X^*, I^*, J^*)$ starting from $(X(T), I(T), J(T))$ by using the same Poisson processes for their graphical constructions. We observe that up to time $(\logn)^2$, $(X,I, J)$ and $(X^*, I^*, J^*)$ coincide. Therefore, this term is $\Ocal{n^{-3}}$ by the strong Markov property at time $T$ and \eqref{bound_tau_from_goodset}, which finishes our proof.
    \end{enumerate}
\end{proof}
Now we can prove Proposition \ref{quick_convergence}.
\begin{proof}[Proof of Proposition \ref{quick_convergence}] 
    Let $t = \alpha \logn$, \orange{where $\alpha$ is as in Proposition \ref{coalescence_logn}}. We say that two configurations are adjacent if they differ only by one jump. \eqref{quick_coalescence} implies that for any $x, y$ such that $x, y$ are adjacent and $\norminfty{x} \vee \norminfty{y} \leq \alpha_1 \logn$, $\dtv{P^t_x}{P^t_y} = \Ocal{n^{-3}}$. Now for $x, y$ arbitrary such that $\norminfty{x} \vee \norminfty{y} \leq \alpha_1\logn$, we can always connect $x$ and $y$ by a path, \ie a sequence $(\omega_0, \omega_1, ..., \omega_k) $ in $\Omega$ such that $\omega_0 = x, \omega_k = y$, and $\omega_{l-1}$ is adjacent to $\omega_l$ for $1 \leq l \leq k$. Furthermore, we can pick one of the shortest paths to make sure that $k\leq m$ and \[\max\limits_{1 \leq l \leq k} \norminfty{\omega_l} \leq \norminfty{x} \vee \norminfty{y} \leq \alpha_1\logn.\] Then by triangle inequality, 
    \[\dtv{P^t_x}{P^t_y} \leq \sum_{u = 1}^k \dtv{P^t_{\omega_{u-1}}}{P^t_{\omega_u}} \leq m \Ocal{n^{-3}} = \Ocal{n^{-2}},\]
    \orange{where the last equality is due to \eqref{densitybounded}.} By stationarity of $\pi$ and convexity of $\dtv{\cdot}{\cdot}$, 
    \begin{align*}
        \dtv{P^t_x}{\pi} &\leq \sum_{y \in \Omega} \pi(y) \dtv{P^t_x}{P^t_y}\\
        &= \sum_{\{y: \norminfty{y} > \alpha_1 \logn\}} \pi(y) \dtv{P^t_x}{P^t_y} + \sum_{\{y: \norminfty{y} \leq \alpha_1 \logn\}} \pi(y) \dtv{P^t_x}{P^t_y}\\
        &\leq \pi(\norminfty{y} > \alpha_1 \logn) + \pi(\norminfty{y} \leq \alpha_1 \logn) \Ocal{n^{-2}}.
    \end{align*}
    Moreover, by letting $t \to \infty$ in Proposition \ref{key_proposition}, we obtain $\pi(\norminfty{y} > \alpha_1 \logn) = \Ocal{n^{-2}}$. Combining it with the above inequality, we deduce the claim.
\end{proof}
\section{The Poincaré constant} 
This section is devoted to proving Theorem \ref{spectral_gap} and Corollary \ref{application_spectral_gap}. We first recall a classical lemma for general Markov processes:
\begin{lemma}[Lower bound on Poincaré constant]\label{spectral_gap_lemma}
    Let $\Omega$ be a finite state space, and let $\Lcal$ be an irreducible reversible Markov generator on $\Omega$. Fix $\gamma > 0$, and suppose that for any $(x, y) \in \Omega \times \Omega$ such that $\Lcal(x, y) > 0$, there exists a coupling $\Pbb_{x, y}$ of two processes with generator $\Lcal$ starting from $x$ and $y$ such that 
    \[\esperancewithstartingpoint{x, y}{e^{\gamma \tau}} < \infty,\]
    where $\tau$ is the coalescence time of the two processes. Then $\lambda_*(\Lcal) > \gamma$.
\end{lemma}
\begin{proof}
    Let $A = \max\limits_{x, y: \Lcal(x, y) > 0} \esperancewithstartingpoint{x, y}{e^{\gamma \tau}}$. Then for any $x, y$ such that $\Lcal(x, y) > 0$, 
    \begin{equation*}
        \dtv{P^t_x}{P^t_y} \leq \probawithstartingpoint{x, y}{\tau > t} \leq e^{-\gamma t}\esperancewithstartingpoint{x, y}{e^{\gamma \tau}} \leq Ae^{-\gamma t} .
    \end{equation*}
    Now for $(x, y) \in \Omega \times \Omega$ arbitrary, as $\Lcal$ is irreducible, we can connect $x$ and $y$ by a path, \ie a sequence $(\omega_0, \omega_1, ..., \omega_k)$ in $\Omega$ such that $\omega_0 = x, \omega_k = y$, and $\Lcal(\omega_{l - 1}, \omega_l) > 0$, for $1 \leq l \leq k$. Picking one of the shortest path ensures that $k \leq |\Omega|$. Hence by the triangle inequality, 
    \begin{equation*}
        \dtv{P^t_x}{P^t_y} \leq \sum_{l = 1}^k \dtv{P^t_{\omega_{l-1}}}{P^t_{\omega_l}} \leq A|\Omega|e^{-\gamma t}.
    \end{equation*}
    By stationarity of $\pi$ and convexity of $\dtv{\cdot}{\cdot}$, 
    \begin{equation*}
        \dtv{P^t_x}{\pi} \leq \sum_{y \in \Omega} \pi(y)\dtv{P^t_x}{P^t_y} \leq A|\Omega| e^{-\gamma t}. 
    \end{equation*}
    We deduce that
    \begin{equation*}
        -\dfrac{1}{t} \log \max_{x \in \Omega} \dtv{P^t_x}{\pi} \geq \gamma - \dfrac{1}{t} (\log A + \log|\Omega|).
    \end{equation*}
    The claim is obtained simply by letting $t \to \infty$.
\end{proof}
Thanks to Lemma \ref{spectral_gap_lemma}, in order to prove $\lambda_* = \Omega(1)$, we just need to prove the following:
\begin{proposition}[Exponential moment of coalescence time]\label{spectral_gap_lower_bound}
    There exists a dimension-free constant $\gamma$ such that for all $(x, i, j) \in \Omega\times [n] \times [n]$, for the coupling of two Zero-Range processes starting from $x + \delta_i,\, x+ \delta_j$ using tagged particles as described above,
    \begin{align}
        \esperancewithstartingpoint{x, i, j}{e^{\gamma \tau}} < \infty.
    \end{align}
\end{proposition}

\begin{proof}
    \orange{We only need to prove the result for big enough $n$. Let $\alpha_1$ be as in Proposition \ref{key_proposition}. Let $\theta_1 = 6/\alpha_1, \beta_1$ be a dimension-free constant, and $L_1 = L(\theta_1, \beta_1)$ as in Lemma \ref{boundL}. Let $\theta_2 < \theta_1$ be a constant that satisfies Proposition \ref{stability}, and let $L_2 > L_1$ and $\good$ be defined as in subsection 3.2.} Let $T_1$ be the hitting time of the set $\{\phi^{\theta_1} \leq L_1\}$. $\eqref{tempsdattentelaplace}$ says that $\esperancewithstartingpoint{x}{e^{\beta_1 T_1}}$ is finite for all $x \in \Omega$, so we only need to prove the result for $x \in \{\phi^{\theta_1} \leq L_1\}$. 
    By abuse of notation, we will note $\Pbb_*(\cdot)$ (resp. $\Ebb_*(\cdot)$) for the maximum of $\Pbb_{x, i, j}(\cdot)$ (resp. $\Ebb_{x, i, j}(\cdot)$) taken over all $x$ such that $\phi^{\theta_1}(x) \leq L_1$ and all $(i, j) \in [n] \times [n]$.
    We will prove that there exists $\gamma$ such that $\esperancewithstartingpoint{*}{e^{\gamma(\tau \wedge k)}}$ is bounded uniformly in $k$. Then the claim is proved simply by letting $k$ tend to infinity using the Monotone Convergence Theorem.\\
    Let $A = \good \cup \{y: \exists x \in \good,\: \Lcal(x, y) > 0\}$. Let $\Tbad$ be the exit time from $\good$, and let $T_2:= \Tbad \wedge \alpha_3 \logn$, where $\alpha_3$ is defined in Corollary \ref{staying_good_coalescence}. In fact, $A$ is the set of all possible values of $X$ up to time $\Tbad$, and in particular, $X(T_2) \in A$.\deleted{ Recall that $\alpha_1 = \dfrac{6}{\theta_1}, \theta_2 \leq \theta_1, L_2 \geq L_1$. Hence} \orange{By our definitions of the dimension free constants, when $n$ is large enough, } $\{\phi^{\theta_1} \leq L_1\} \subset \{\phi^{\theta_2} \leq L_2\} \cap \{\norminfty{\cdot} \leq \alpha_1\logn\}$. Consequently, by Proposition \ref{stability} and Proposition \ref{coalescence_logn},
    \begin{equation*}
        \probawithstartingpoint{*}{\tau \geq T_2} = \Ocal{n^{-3}}.
    \end{equation*}
    Note that $T_2 \leq \alpha_3 \logn$, and hence $\esperancewithstartingpoint{*}{e^{\gamma(\tau\wedge k)} \indicator{\tau < T_2}} \leq e^{\gamma \alpha_3 \logn} = n^{\gamma\alpha_3}$. We deduce that 
    \begin{align*}
        \esperancewithstartingpoint{*}{e^{\gamma (\tau\wedge k)}} \leq n^{\gamma \alpha_3} + \esperancewithstartingpoint{*}{e^{\gamma(\tau\wedge k)}\indicator{\tau \geq T_2}}.
    \end{align*}
    Conditionally on $\Fcal_{T_2}$, by the strong Markov property, we have
    \begin{align*}
        \esperancewithstartingpoint{*}{e^{\gamma(\tau\wedge k)}\indicator{\tau \geq T_2}} &\leq \esperancewithstartingpoint{*}{e^{\gamma(T_2 \wedge k)}\indicator{\tau \geq T_2}\esperancewithstartingpoint{X(T_2)}{e^{\gamma(\tau \wedge k)}}}\\
        &\leq n^{\gamma\alpha_3} \probawithstartingpoint{*}{\tau \geq T_2} \max_{y \in A} \esperancewithstartingpoint{y}{e^{\gamma T_1}} \esperancewithstartingpoint{*}{e^{\gamma(\tau \wedge k)}},
    \end{align*}    
    \orange{where in the last inequality we use the fact that $X(T_2) \in A$, almost surely.} Putting things together, we obtain 
    \begin{equation}
        \esperancewithstartingpoint{*}{e^{\gamma(\tau \wedge k)}} \leq n^{\gamma\alpha_3} + \Ocal{n^{\gamma\alpha_3 - 3}} \max_{y \in A} \esperancewithstartingpoint{y}{e^{\gamma T_1}} \esperancewithstartingpoint{*}{e^{\gamma(\tau \wedge k)}}.
    \end{equation}
    Consequently, 
    \begin{equation*}
        \esperancewithstartingpoint{*}{e^{\gamma(\tau \wedge k)}} \leq \dfrac{n^{\gamma\alpha_3}}{1 - \Ocal{n^{\gamma\alpha_3 - 3}} \max\limits_{y\in A}\esperancewithstartingpoint{y}{e^{\gamma T_1}}},
    \end{equation*}
    provided that the denominator of the right-hand side is positive. Note that for $y \in A$, $\norminfty{y} = \Ocal{\logn}$, and hence $\phi^{\theta_1}(y) < n^p$, for some dimension-free constant $p$. Then by \eqref{tempsdattentelaplace} and Jensen's inequality, for all $y\in A$ and $\gamma < \beta_1$, 
    \begin{equation*}
        \esperancewithstartingpoint{y}{e^{\gamma T_1}} \leq \esperancewithstartingpoint{y}{e^{\beta_1 T_1}}^{\gamma/\beta_1} = \Ocal{n^{\gamma p/\beta_1}}.
    \end{equation*}
    Then the denominator above is $1 - \Ocal{n^{\gamma\alpha_3 + \gamma p/\beta_1 - 3 }}$, which is positive when $\gamma$ is small enough. This finishes our proof.
\end{proof}
We now prove Theorem \ref{spectral_gap} and Corollary \ref{application_spectral_gap}.
\begin{proof}[Proof of Theorem \ref{spectral_gap} and Corollary \ref{application_spectral_gap}]
    In \cite{hermon2019version} (more precisely, in Corollary 3 and Lemma 13), the authors prove that for $P$ a doubly stochastic transition matrix on $[n]$,  
    \begin{equation*}
        \lambda_*(\Lcal) \leq \dfrac{\lambda_*(\Lcal^P)}{\lambda_*(P)} \leq (1 - 1/n)\dfrac{\esperance{r(X_1)}}{\variance{X_1}},
    \end{equation*}
    where the expectation and the variance are taken with respect to the stationary law $\pi$ of $\Lcal$.
    Lemma \ref{spectral_gap_lemma} and Proposition \ref{spectral_gap_lower_bound} readily imply that $\lambda_*(\Lcal) = \Omega(1)$, and hence so is $\dfrac{\lambda_*(\Lcal^P)}{\lambda_*(P)}$.
    It remains to prove that 
    \begin{equation}\label{interesting_quantity}
        \dfrac{\esperance{r(X_1)}}{\textrm{Var}\left[X_1\right]} = \Ocal{1}.
    \end{equation}
     We consider two cases: the case where the density is bounded away from zero : $\dfrac{1}{2} \leq \esperance{X_1} = \dfrac{m}{n} \leq \rho$, and the case of low density: $\esperance{X_1} = \dfrac{m}{n} < \dfrac{1}{2}$.\\
    In case the density is bounded away from zero, it is easy to deduce from Proposition \ref{geometricdecay} that $\esperance{r(X_1)} = \Theta(1)$ and $\variance{X_1} = \Theta(1)$, and hence the claim.\\
    In the case of low density, we have $\variance{X_1} \geq \esperance{X_1} - \esperance{X_1}^2$ and $\esperance{r(X_1)} \leq \sup\limits_{k\in\Zbb_+} \dfrac{r(k)}{k} \esperance{X_1}$. We deduce that
    \[\dfrac{\esperance{r(X_1)}}{\variance{X_1}} \leq \dfrac{\sup\limits_{k\in\Zbb_+} \dfrac{r(k)}{k}}{1 - \esperance{X_1}} < 2\sup\limits_{k\in\Zbb_+} \dfrac{r(k)}{k},\]
    which finishes the proof.
\end{proof}
Finally, we compute the Poincaré constant of the transition matrix $P$ in Example \ref{torus_example}.
\begin{proof}[Proof of example \ref{torus_example}]
    Let $P_1$ be the transition matrix of the simple random walk on $\Zbb/p\Zbb$. It is not hard to see that both $P$ and $P_1$ are reversible \wrt the uniform measures on their domains. Note that, for $f_1, ..., f_d$ some eigenfunctions of $P_1$ with eigenvalues $\lambda_1, ..., \lambda_d$ respectively, the function $f: (\Zbb/p\Zbb)^d \to \Rbb $ defined by \begin{equation*}
        f(x_1, ..., x_d) = \prod_{i = 1}^d f_i(x_i)
    \end{equation*}
    is an eigenfunction of $P$ with eigenvalue $\lambda := \dfrac{\sum_{i = 1}^d \lambda_i}{d}$. Moreover, as the eigenfunctions of $P_1$ generate the space of functions from $\Zbb/p\Zbb$ to $\Rbb$ due to reversibility, the functions $f$ of the form above also generate the space of functions from $(\Zbb/p\Zbb)^d$ to $\Rbb$. Hence every eigenvalue of $P$ is of the form $\lambda = \dfrac{\sum_{i = 1}^d \lambda_i}{d}$, where $\lambda_i, \; 1\leq i \leq d,$ are some eigenvalues of $P_1$. It is well-known that the eigenvalues of $P_1$ are $\cos\left(\dfrac{2\pi k}{p}\right), \; 0\leq k \leq p$ (see \eg Chapter 12 of \cite{wilmer2009markov}). Due to reversibility of $P$, $\lambda_*(P)$ is the smallest eigenvalue of $I-P$ (see \eg Chapter 12 of \cite{wilmer2009markov}), so it is given by 
    \begin{equation*}
        \lambda_*(P) = 1 - \dfrac{1}{d}\left(\cos\left(\dfrac{2\pi }{p}\right) + d-1\right) = \dfrac{1}{d}\left(1 - \cos\left(\dfrac{2\pi }{p}\right)\right) \approx \dfrac{1}{d} \cdot\dfrac{2\pi^2}{p^2},
    \end{equation*}
    which finishes our proof.
\end{proof}

\bibliographystyle{abbrv}
\bibliography{ref}

\begin{thebibliography}{10}

\bibitem{boucheron:hal-00942704}
S.~Boucheron, G.~Lugosi, and P.~Massart.
\newblock {\em Concentration inequalities: A nonasymptotic theory of
  independence}.
\newblock Oxford University Press, Oxford, 2013.

\bibitem{bubley1997path}
R.~Bubley and M.~Dyer.
\newblock Path coupling: A technique for proving rapid mixing in markov chains.
\newblock In {\em Proceedings 38th Annual Symposium on Foundations of Computer
  Science}, pages 223--231, Miami Beach, FL, USA, 1997. IEEE.

\bibitem{caputo2004spectral}
P.~Caputo.
\newblock Spectral gap inequalities in product spaces with conservation laws.
\newblock In {\em Stochastic analysis on large scale interacting systems},
  volume~39 of {\em Adv. Stud. Pure Math.}, pages 53--88. Math. Soc. Japan,
  Tokyo, 2004.

\bibitem{EK}
S.~N. Ethier and T.~G. Kurtz.
\newblock {\em Markov processes -- characterization and convergence}.
\newblock Wiley Series in Probability and Mathematical Statistics: Probability
  and Mathematical Statistics. John Wiley \& Sons Inc., New York, 1986.

\bibitem{hermon2019version}
J.~Hermon and J.~Salez.
\newblock A version of {A}ldous' spectral-gap conjecture for the zero range
  process.
\newblock {\em The Annals of Applied Probability}, 29(4):2217--2229, 2019.

\bibitem{hermon2020cutoff}
J.~Hermon and J.~Salez.
\newblock Cutoff for the mean-field zero-range process with bounded monotone
  rates.
\newblock {\em The Annals of Probability}, 48(2):742--759, 2020.

\bibitem{lacoin2016cutoff}
H.~Lacoin.
\newblock The cutoff profile for the simple exclusion process on the circle.
\newblock {\em The Annals of Probability}, 44(5):3399--3430, 2016.

\bibitem{lacoin2016mixing}
H.~Lacoin.
\newblock Mixing time and cutoff for the adjacent transposition shuffle and the
  simple exclusion.
\newblock {\em The Annals of Probability}, 44(2):1426--1487, 2016.

\bibitem{wilmer2009markov}
D.~A. Levin, Y.~Peres, and E.~L. Wilmer.
\newblock {\em Markov chains and mixing times}.
\newblock American Mathematical Society, Providence, RI, 2009.
\newblock With a chapter by James G. Propp and David B. Wilson.

\bibitem{merle2018cutoff}
M.~Merle and J.~Salez.
\newblock Cutoff for the mean-field zero-range process.
\newblock {\em The Annals of Probability}, 47(5):3170--3201, 2019.

\bibitem{morris2006}
B.~Morris.
\newblock Spectral gap for the zero range process with constant rate.
\newblock {\em The Annals of Probability}, 34(5):1645--1664, 2006.

\bibitem{Pro2005}
P.~E. Protter.
\newblock {\em Stochastic integration and differential equations}, volume~21 of
  {\em Stochastic Modelling and Applied Probability}.
\newblock Springer-Verlag, Berlin, 2005.
\newblock Second edition. Version 2.1, Corrected third printing.

\bibitem{vandegeer1995}
S.~van~de Geer.
\newblock Exponential inequalities for martingales, with application to maximum
  likelihood estimation for counting processes.
\newblock {\em The Annals of Statistics}, 23(5):1779--1801, 1995.

\end{thebibliography}

\end{document}